\documentclass{amsart}

\usepackage{amsmath}
\usepackage{amsfonts}
\usepackage{amssymb}
\usepackage{comment}
\usepackage{epsfig}
\usepackage{psfrag}
\usepackage{mathrsfs}
\usepackage{amscd}
\usepackage[all]{xy}
\usepackage{rotating}
\usepackage{lscape}
\usepackage{amsbsy}
\usepackage{verbatim}
\usepackage{moreverb}
\usepackage{epigraph}

\newtheorem{theorem}{Theorem}[section]
\newtheorem{prop}[theorem]{Proposition}
\newtheorem{lemma}[theorem]{Lemma}
\newtheorem{cor}[theorem]{Corollary}

\newtheorem{definition}[theorem]{Definition}

\theoremstyle{remark}
\newtheorem{remark}[theorem]{Remark}
\newtheorem{example}[theorem]{Example}
\newtheorem{q}[theorem]{Question}

\DeclareMathOperator{\colim}{colim}

\DeclareMathOperator{\Hom}{Hom}
\DeclareMathOperator{\Fun}{Fun}

\DeclareMathOperator{\Map}{Map}

\DeclareMathOperator{\Spec}{Spec}

\DeclareMathOperator{\Mod}{Mod}

\DeclareMathOperator{\qc}{QC}

\def\1bord{1\mathrm{Bord}}
\def\2bord{2\mathrm{Bord}}
\def\3bord{3\mathrm{Bord}}

\DeclareMathOperator{\hh}{HH}

\newcommand{\ra}{\rightarrow}
\newcommand{\la}{\leftarrow}

\def\cC{\mathcal C}\def\cD{\mathcal D}
\def\cE{\mathcal E}\def\cF{\mathcal F}\def\cH{\mathcal H}
\def\cL{\mathcal L}
\def\cM{\mathcal M}\def\cO{\mathcal O}\def\cP{\mathcal P}
\def\cS{\mathcal S}

\def\cZ{\mathcal Z}

\def\QQ{\mathbb Q}

\def\ZZ{\mathbb Z}

\def\ot{\otimes}

\def\ot{\otimes}
\def\XYX{X \times_Y X}

\def\fz{\mathfrak z}

\def\ftr{\mathfrak {tr}}

\def\D{{\mathcal D}}

\newcommand{\on}{\operatorname}
\newcommand{\Vect}{\on{Vect}}
\newcommand{\Rep}{\on{Rep}}
\newcommand{\module}{-\on{mod}}
\newcommand{\Ind}{\on{Ind}}
\newcommand{\Perf}{\on{Perf}}

\newcommand{\ti}{\times}

\newcommand{\Id}{{\rm id}}
\newcommand{\intHom}{{\mathcal Hom}}
\newcommand{\intEnd}{{\mathcal End}}

\newcommand{\aff}{\it{aff}}
\newcommand{\mc}{\mathcal}
\def\Aff{\it{Aff}}
\def\risom{\stackrel{\sim}{\ra}}

\def\oo{\infty}

\def\Tr{\mathcal Tr}

\newcommand{\bs}{\backslash}

\newcommand{\Kar}{{st}}

\newcommand\Alg{\mathcal Alg}
\newcommand\Top{\mathcal Top}

\newcommand\ind{\varinjlim}

\begin{document}

\title{integral transforms and Drinfeld centers\\
in derived algebraic geometry}
\author{David Ben-Zvi, John Francis, and David Nadler}

\address{Department of Mathematics\\University of Texas\\Austin, TX 78712-0257}
\email{benzvi@math.utexas.edu}
\address{Department of Mathematics\\Northwestern University\\Evanston, IL 60208-2370}
\email{jnkf@math.northwestern.edu}
\address{Department of Mathematics\\Northwestern University\\Evanston, IL 60208-2370}
\email{nadler@math.northwestern.edu}

\maketitle

\epigraph{\it{Compact objects are as necessary to this subject as
air to breathe}.} {R.W. Thomason to A. Neeman, \cite{neemanbook}}

\begin{abstract}
We study the interaction between geometric operations on stacks and
algebraic operations on their categories of sheaves.  We work in the
general setting of derived algebraic geometry: our basic objects are
derived stacks $X$ and their $\oo$-categories $\qc(X)$ of
quasi-coherent sheaves. (When $X$ is a familiar scheme or stack,
$\qc(X)$ is an enriched version of the usual quasi-coherent derived
category $D_{qc}(X)$.)  We show that for a broad class of derived
stacks, called perfect stacks, algebraic and geometric operations on
their categories of sheaves are compatible.  We identify the category
of sheaves on a fiber product with the tensor product of the
categories of sheaves on the factors.  We also identify the category
of sheaves on a fiber product with functors between the categories of
sheaves on the factors (thus realizing functors as integral
transforms, generalizing a theorem of To\"en~\cite{Toen dg} for
ordinary schemes).  As a first application, for a perfect stack $X$,
consider $\qc(X)$ with its usual monoidal tensor product.  Then our
main results imply the equivalence of the Drinfeld center (or
Hochschild cohomology category) of $\qc(X)$, the trace (or Hochschild
homology category) of $\qc(X)$ and the category of sheaves on the loop
space of $X$. More generally, we show that the $\cE_n$-center and the
$\cE_n$-trace (or $\cE_n$-Hochschild cohomology and homology
categories respectively) of $\qc(X)$ are equivalent to the category of
sheaves on the space of maps from the $n$-sphere into $X$. This
directly verifies geometric instances of the categorified Deligne and
Kontsevich conjectures on the structure of Hochschild cohomology. As a
second application, we use our main results to calculate the Drinfeld
center of categories of linear endofunctors of categories of
sheaves. This provides concrete applications to the structure of Hecke
algebras in geometric representation theory.  Finally, we explain how
the above results can be interpreted in the context of
topological field theory.
\end{abstract}

\tableofcontents



\section{Introduction}
This paper is devoted to the study of natural algebraic operations on
derived categories arising in algebraic geometry.
Our main goal is to identify the category of sheaves on a fiber
product with two algebraic constructions: on the one hand, the tensor product of the categories of sheaves on the factors, and on
the other hand, the category of linear functors between the categories of sheaves on the factors
(thereby realizing functors as integral transforms). 
Among the varied applications of our main results are the calculation of Drinfeld centers
(and higher centers) of monoidal categories of sheaves and the construction of 
topological field theories.

For such questions about the linear algebra of derived categories to be tractable, we work in an enriched setting where we 
replace triangulated categories with a more refined homotopy theory of categories.
Over a ring $k$ of
characteristic zero, one solution is provided by
pre-triangulated differential graded categories
(see for example the survey \cite{Keller}, as well as
\cite{Drinfeld, Toen dg}.) 
The theory of stable $\oo$-categories as
presented in~\cite{topos, dag1} (building on the quasi-categories of~\cite{Joyal}) provides a more general solution. 
We adopt this formalism thanks to the comprehensive foundations
available in~\cite{topos, dag1, dag2, dag3}. 
 In Section~\ref{prelim} below,
we provide a recapitulation of this theory
sufficient for the arguments of this paper.
The reader interested in characteristic zero
applications may consistently substitute differential graded categories 
for $\infty$-categories throughout.

To any scheme or stack $X$, we can assign a
stable $\infty$-category $\qc(X)$ which has the usual unbounded
quasi-coherent derived category $D_{qc}(X)$ as its homotopy category. Tensor product of sheaves provides $\qc(X)$ with the structure of a {symmetric monoidal stable $\infty$-category}.  Our aim is 
to calculate the $\infty$-categories built out of $\qc(X)$ by taking tensor products, linear functors, and more intricate algebraic constructions in terms of the geometry of $X$.

More generally, instead of an ordinary scheme or stack, our starting
point will be a derived stack in the sense of derived algebraic
geometry as developed in~\cite{dag, dag1, dag2, dag3, HAG1, HAG2} (see
also~\cite{Toen} for a concise survey, and Section \ref{prelim} below
for a brief primer).  Recall that stacks and higher stacks arise
naturally from performing quotients (and more complicated colimits) on
schemes.  Thus we correct the notion of forming quotients by passing
to stacks.  Likewise, derived stacks arise naturally from taking fiber
products (and more complicated limits) on schemes and stacks.  Thus we
correct the notion of imposing an equation by passing to derived
stacks.  It is worth mentioning that they also arise naturally in more
general contexts for doing algebraic geometry that are important for
stable homotopy theory.

Recall that the functor of points of a scheme assigns a set to any
commutative ring. Roughly speaking, a derived stack assigns a
topological space to any derived commutative ring (for precise
definitions, see Section \ref{prelim} below).  There are many
variations on what one could mean by a commutative derived ring as
discussed in Section~\ref{derived AG}. For example, in characteristic zero,
one can work with (connective, or homological) commutative
differential graded algebras.  A more general context, which we
typically adopt, is the formalism of connective $\cE_\oo$-ring
spectra.  Since the techniques of this paper apply in all of the usual
contexts, we will often not specify our context further, and use the
term commutative derived ring as a catch-all for any of them.

To any derived stack $X$, we can assign a stable $\infty$-category $\qc(X)$ 
extending the definition for ordinary schemes and stacks. In particular, for an affine derived scheme $X=\Spec R$, 
by definition $\qc(X)$ is the $\infty$-category of $R$-modules
$\Mod_R$ whose  homotopy category is the usual unbounded derived category of $R$-modules.
Tensor product provides $\qc(X)$ with the structure of a {symmetric monoidal stable $\infty$-category}.

Our main technical results are two algebraic identifications of the
$\infty$-category of quasi-coherent sheaves on a
fiber product: first, we identify it with the tensor product of the
$\infty$-categories of sheaves on the factors over the $\infty$-category of sheaves on the base; 
second, we identify it with the $\infty$-category of functors between the $\infty$-categories of sheaves on the factors that 
are linear over the $\oo$-category of sheaves on the base
(thereby realizing functors as integral transforms). 
Our results hold for a broad class of stacks, called perfect stacks, which we introduce in the next section
immediately below.

Our main applications are the calculation of the Drinfeld centers
(and higher $\cE_n$-centers)
of $\infty$-categories of sheaves and functors. For example,
we identify the Drinfeld center of the quasi-coherent affine Hecke category
with sheaves on the moduli of local systems on a torus.
We also explain how all of our results fit into the framework
of 3-dimensional topological field theory (specifically of Rozanksy-Witten type). 
In particular, we verify categorified analogues
of the Deligne and Kontsevich conjectures on the $\cE_n$-structure of Hochschild cohomology.

 
\subsection{Perfect stacks}

For an arbitrary derived stack $X$, the $\infty$-category $\qc(X)$ is difficult to control
algebraically. For example, it may contain large objects that are impossible to construct in terms of concrete, locally-finite objects.

To get a handle on $\qc(X)$, we need to know that it has a small $\infty$-subcategory $\qc(X)^\circ$ of ``generators" which are
``finite" in an appropriate sense.
There are two common notions of when
a small subcategory $\cC^{\circ}$
generates a category $\cC$, and they have natural $\oo$-analogues.
On the one hand, we could ask that $\cC$ be the inductive limit or ind-category $\Ind \cC^\circ$ (i.e., that $\cC$ is freely
generated from $\cC^\circ$ by taking inductive limits), and on the 
other hand, we could ask that in $\cC$ the right orthogonal of $\cC^{\circ}$ vanishes.
When $\cC$ is $\qc(X)$, there are three common notions of which objects may be considered finite: perfect objects, dualizable objects, and compact objects, which refer respectively to the geometry, monoidal structure, and categorical structure of $\qc(X)$.
We review all of these notions and the relations between them in Section \ref{air section}
(in particular, perfect and dualizable objects always coincide).


\medskip

To have a tractable and broadly applicable class of derived stacks, 
we introduce the notion of a {perfect stack}. 
 By definition, an object $M$ of
$\qc(X)$ is a perfect complex if locally for any affine $U\to X$ its restriction to $U$ is a perfect module (finite complex 
of projective modules).
Equivalently, $M$ is dualizable with respect to the monoidal structure on $\qc(X)$.
A derived stack $X$ is said to be perfect if
it has affine diagonal and the $\infty$-category $\qc(X)$ is the inductive limit
$$
\qc(X) \simeq \Ind\Perf(X)
$$
of the full $\infty$-subcategory $\Perf(X)$ of perfect complexes.
A morphism $X\to Y$
is said to be perfect if  its fibers $X\ti_Y U$ over affines $U\to Y$ are perfect.

On a perfect stack $X$, compact objects of $\qc(X)$ are the same thing as perfect complexes
(which in turn are always the same thing as dualizable objects).
In fact, we have the following alternative formulation: a derived stack $X$ is perfect
if it has affine diagonal, $\qc(X)$ is compactly generated (that is, there is no right orthogonal
to the compact objects), and compact objects and dualizable objects coincide.

\begin{remark}

The notion of compactly generated categories is a standard one in
homotopy theory, especially in conjunction with themes such as Brown
representability and Bousfield localization.
Schwede and Shipley
\cite{schwedeshipley} prove in great generality that
compactly generated categories can be expressed as categories of modules.

In algebraic geometry, the
importance of the interplay between compact
and perfect objects
was originally recognized and put to great use by Thomason
\cite{thomasontrobaugh}. These ideas were combined with
homotopical techniques by B\"okstedt and Neeman
\cite{BoN}, and further developed and enhanced by Neeman
\cite{neemanTTY,neeman} and many others \cite{Keller,BvdB, Toen dg}.
The key property of derived categories of quasi-coherent sheaves
on quasi-compact, separated schemes identified in these papers is
that on the one hand, they are compactly generated, and on the other
hand, their compact and perfect objects coincide. 
Such categories
appear as {\em unital algebraic stable homotopy categories} in the
general axiomatic framework developed by Hovey, Palmieri and
Strickland \cite{HPS}. This combination of properties 
underlies the definition of a perfect stack.



\end{remark}

\medskip

In Section \ref{base change}, we establish base change and
the projection formula for perfect morphisms following arguments in Lurie's
thesis~\cite{dag} (in fact, the arguments here only use that the structure
sheaf is relatively compact).

In Section~\ref{sect properties of air}, we show that the class of perfect stacks is very broad.
We show that the following are all perfect stacks:

\begin{enumerate}

\item Quasi-compact derived schemes with affine diagonal (following arguments of \cite{neeman}).

\item The total space of a quasi-projective morphism over a perfect base.

\item In characteristic zero, 
the quotient $X/G$ of a quasi-projective derived scheme $X$ 
by a linear action of an affine group $G$.

\item The quotient $X/G$ of a perfect stack $X$ by a
  finite affine group scheme in ``good" characteristics for
  $G$.


\item The mapping stack $X^\Sigma=\Map(\Sigma, X)$, 
for a perfect stack $X$ and a finite simplicial set $\Sigma$.

\item Fiber products of perfect stacks.

\end{enumerate}

We also show that any morphism $X\to Y$ between perfect stacks is itself perfect.
%

Though the above examples show that perfect stacks cover a broad array
of spaces of interest, it is worth pointing out that there are many
commonly arising derived stacks that are imperfect. Since categories
of quasi-coherent sheaves are usually compactly generated, the typical
reason for $X$ to be imperfect is that the structure sheaf $\cO_X$ (which is
always dualizable) fails to be compact (or equivalently, the global
sections functor fails to preserve colimits). This can occur if the
cohomology $\Gamma(X, \cO_X)$ is too big such as for (1) the
classifying space of a finite group in modular characteristic (for the simplest example,
one can take $B\ZZ/2\ZZ$ when $2$ is not invertible), (2) the classifying space of a
topological group such as $S^1$, or (2) an ind-scheme such as the formal
disc $\on{Spf} k[[t]]$. In the opposite direction, categories of
$\D$-modules on smooth schemes (which can be considered as quasi-coherent sheaves
on the corresponding de Rham stacks) have compact unit, but also many
compact objects (such as $\D$ itself) which are not dualizable.

\subsection{Tensors and functors}

For $X, X'$ ordinary schemes over a ring $k$, a theorem of To\"en
\cite{Toen dg} identifies the dg category of $k$-linear continuous
(that is, colimit preserving)
functors 
with the dg category of integral kernels
$$
\Fun_k(\qc(X), \qc(X'))\simeq \qc(X \times_k X').
$$ 
For dg categories of perfect (equivalently, bounded coherent) complexes on
smooth projective varieties, an analogous result was proved by Bondal,
Larsen and Lunts \cite{BLL} as well as by To\"en \cite{Toen dg} (generalizing
Orlov's theorem \cite{Orlov} characterizing {equivalences} as
Fourier-Mukai transforms).

We've collected our main technical results in the following generalization.
In the statement,
the tensors and functors of $\infty$-categories of quasi-coherent sheaves are
calculated in the symmetric monoidal $\infty$-category $\mc Pr^{\rm L}$
of presentable $\infty$-categories with morphisms left adjoints
(as developed in~\cite[4]{dag2} and~\cite[5]{dag3}, see Section \ref{prelim} for a precise summary).
The tensors and functors of $\infty$-categories of perfect complexes are
calculated in the symmetric monoidal $\infty$-category $\Kar$
of $k$-linear idempotent complete stable small $\infty$-categories
(as developed in Section~\ref{tensors of cats} below).

\begin{theorem}\label{stating main result}

\begin{enumerate}

\item For $X\ra Y\la X'$ maps of perfect stacks, there is a canonical equivalence
$$\qc(X\times_Y X')\simeq \qc(X)\ot_{\qc(Y)}\qc(X')$$
between the $\infty$-category of sheaves on the derived fiber product and the tensor
product of the $\infty$-categories of sheaves on the factors.

There is also a canonical equivalence
$$\Perf(X\times X')\simeq \Perf(X)\ot \Perf(X')$$
for $\infty$-categories of perfect complexes.

\item For $X\to Y$ a perfect morphism to a derived stack $Y$ with affine diagonal,
and $X'\to Y$ arbitrary, there is a
canonical equivalence
$$\qc(X\times_Y X')\simeq \Fun_{\qc(Y)}(\qc(X),\qc(X'))$$
between the $\infty$-category of sheaves on the derived fiber product and 
the $\infty$-category of colimit-preserving $\qc(Y)$-linear functors.

When $X$ is a smooth and proper perfect stack, there is also 
a canonical equivalence
$$\Perf(X\times X')\simeq \Fun(\Perf(X), \Perf(X'))$$
for $\infty$-categories of perfect complexes.

\end{enumerate}
\end{theorem}

Our arguments in Section~\ref{integral} first establish
the pair of assertions about tensors, then deduce sufficient duality to
conclude the pair of assertions about functors.

\begin{remark} The hypothesis that our stacks are perfect appears to be essential, and we do not expect the above theorem to hold in significantly greater generality. 
Alternatively, in complete generality, one should rather replace the
notion of tensor product. Jacob Lurie has described (in private
communication) a completed tensor product for stable
$\infty$-categories equipped with $t$-structures, and such that
passing to quasi-coherent sheaves takes fiber products of geometric
stacks to the completed tensor product of $\infty$-categories.
\end{remark}


\subsection{Centers and traces}\label{intro centers}
A basic operation on associative algebras is the calculation of their center.
The derived version of the center of an associative algebra $A$ is the
Hochschild cochain complex (or simply the Hochschild cohomology), which calculates the
derived endomorphisms of $A$ as an $A$-bimodule. 
Another basic operation is the calculation of the universal trace (i.e., the universal
target for a map out of $A$ coequalizing left and right multiplication).
The derived version of the universal trace is the Hochschild chain complex (or the
Hochschild homology),
which calculates the derived tensor product of $A$ with itself as an $A$-bimodule.
%

In Section~\ref{defining centers},
we extend
the notion of Hochschild homology and cohomology to associative (or $\cE_1$-)algebra objects in 
arbitrary closed symmetric monoidal $\infty$-categories.
(As with any structure in an $\oo$-category, 
an associative multiplication, or $\cE_1$-structure, is a homotopy coherent notion.) 
In the case of chain complexes,
we recover the usual Hochschild chain and cochain complexes.
In the case of spectra, 
we recover topological Hochschild homology and cohomology.


\begin{definition}
Let $A$ be an associative algebra object in a closed symmetric monoidal
$\oo$-category~$\cS$.

\begin{enumerate}

\item The derived center or Hochschild cohomology $\cZ(A)=\hh^*(A)\in \cS$
is the endomorphism object $\intEnd_{A\ot A^{\rm op}}(A)$ of $A$ as an $A$-bimodule.

\item The derived trace or Hochschild homology $\Tr(A)=\hh_*(A)\in \cS$ is 
the pairing object ${A\ot_{A\ot A^{\rm op}} A}$
 of $A$ with itself as an $A$-bimodule.

\end{enumerate} 
\end{definition}
We show in particular that $\hh^*(A)$ and $\hh_*(A)$ are calculated in this generality
by a version of the usual Hochschild complexes, the cyclic bar construction.

We apply this definition in the following setting.
We will take $\cS$ to be the $\oo$-category
${\mc Pr}^{\rm L}$ of presentable $\oo$-categories with morphisms left adjoints.
Then an associative algebra
object in $\mc Pr^{\rm L}$ is a monoidal presentable $\infty$-category $\cC$. 
Thus we have the notion of its 
center (or Hochschild cohomology category) 
and trace (or
Hochschild homology category)
$$
\cZ(\cC)=\Fun_{\cC\otimes \cC^{\rm op}} (\cC , \cC)
\qquad
\Tr(\cC) =\cC\ot_{\cC\ot \cC^{\rm op}} \cC.
$$
These are again presentable $\infty$-categories, 
or in other words, objects of the $\oo$-category $\mc Pr^{\rm L}$.
The center $\cZ(\cC)$ comes equipped with a universal central functor to $\cC$,
and the trace $\Tr(\cC)$ receives a universal trace functor from $\cC$.

\begin{remark}
The above notion of center provides a derived version of the Drinfeld center
of a monoidal category (as defined in \cite{JS}).
To appreciate the difference, consider the abelian tensor category $R\module$ of modules over a
(discrete) commutative ring $R$. Its classical Drinfeld center is $R\module$ again since there are no nontrivial $R$-linear braidings for $R$-modules.
But as we will see below, the derived center of the $\oo$-category of $R$-modules
is the $\oo$-category of modules over the Hochschild chain complex of $R$.
\end{remark}

\begin{remark}
It is important not to confuse the center $\cZ(\cC)$
of a monoidal $\infty$-category $\cC$ with the endomorphisms of the identity functor of 
the underlying $\infty$-category. The former is again an $\infty$-category depending
on the monoidal structure of $\cC$, while the latter
is an algebra with no relation to the monoidal structure of $\cC$. For example if $\cC=A\module$ is modules
over an associative of $A_\oo$-algebra $A$, then the endomorphisms of the identity of $\cC$
are calculated by the (topological) Hochschild cohomology of $A$, not of the $\oo$-category $\cC$.
\end{remark}

Now let us return to a geometric setting and consider a perfect stack $X$
and the presentable $\infty$-category $\qc(X)$. Since $\qc(X)$ is symmetric monoidal,
it defines a commutative (or $\cE_\oo$-)algebra object in $\mc Pr^{\rm L}$, and so in particular, 
an associative algebra object.

To calculate the center and trace of $\qc(X)$,
we introduce the loop space
$$
\cL X=\Map(S^1, X) \simeq X \times_{X\times X} X
$$
where the derived fiber product is along two copies of the diagonal
map. 

For example, when $X$ is an ordinary smooth scheme over a field of characteristic zero, the loop space $\cL X$ is 
the total space $T_X[-1] =\Spec \on{Sym} \Omega_X[1]$ of the shifted tangent bundle of $X$.
When $X$ is the classifying space $BG$ of a group $G$, the loop space
$\cL BG$ is the adjoint quotient $G/G$.

In Section~\ref{defining centers}, as a corollary of our main technical results,
we obtain the following.

\begin{theorem} For a perfect stack $X$,
there are canonical equivalences
$$\cZ(\qc(X))\simeq \qc(\cL X) \simeq \Tr(\qc(X))
$$ between its center, trace, and the $\oo$-category of sheaves on its
loop space.
\end{theorem}

Note that the theorem in particular identifies the Hochschild homology
and cohomology objects associated to the monoidal $\oo$-category
$\qc(X)$. While such an identification may initially appear surprising, it is a natural
consequence of the self-duality of $\qc(X)$ over $\qc(X \ti X)$
which in turn  is a simple consequence of
Theorem \ref{stating main result}.

\begin{remark}
The theorem is a direct generalization of a result of
Hinich~\cite{Hinich}.  He proves that for $X$ a Deligne-Mumford stack
admitting an affine orbifold chart, the Drinfeld center of the
(abelian) tensor category of quasi-coherent sheaves on $X$ is
equivalent to the category of quasi-coherent sheaves on the inertia
orbifold of $X$ (with braided monoidal structure coming from
convolution).  One can recover this from the above theorem by passing
to the hearts of the natural $t$-structures.
\end{remark}

In Section~\ref{sect higher centers}, we also discuss a generalization
of the theorem to $\cE_n$-centers and traces when $n>1$. First, we
introduce the notion of center and trace for an $\cE_n$-algebra object
in $\mc Pr^{\rm L}$, for any $n$. Since $\qc(X)$ is an
$\cE_\oo$-algebra object in $\mc Pr^{\rm L}$, it is also an
$\cE_n$-algebra object, for any $n$.

We show that the $\cE_n$-center and trace of $\qc(X)$ are equivalent
to the $\oo$-category of quasi-coherent sheaves on the derived mapping space
$$
X^{S^n} = \Map(S^n, X).
$$

For example, when $X$ is an ordinary smooth scheme over a field of characteristic zero, the $n$-sphere space $X^{S^n}$ is 
the total space $T_X[-n] =\Spec \on{Sym} \Omega_X[n]$ 
of  the shifted tangent bundle of $X$.
In particular, as we vary $n$,
the $\cE_n$-centers and traces of $\qc(X)$ differ from $\qc(X)$, though they
all have $t$-structures with heart the abelian category of
quasi-coherent sheaves on $X$.

When $X$ is the classifying space $BG$ of a group $G$, the $n$-sphere space $BG^{S^n}$
 can be interpreted as the derived stack 
 ${\mc
Loc}_G(S^n)$ of $G$-local systems on $S^n$. 
In particular, when $n=2$, the $\cE_2$-center and trace of $\qc(BG)$
is the $\infty$-category $\qc({\mc Loc}_G(S^2))$ which 
appears in the Geometric Langlands program
(see for example, the work of
Bezrukavnikov and Finkelberg \cite{BezFink} 
who identify  $\qc({\mc
Loc}_G(S^2))$ with the derived Satake
(or spherical Hecke) category of arc-group equivariant constructible sheaves
on the affine Grassmannian for the dual group).


\subsection{Hecke categories}
Let $X\to Y$ be a map of perfect stacks.
Our main technical result gives an identification of (not necessarily symmetric)
monoidal $\infty$-categories
$$
\qc(\XYX) \simeq \Fun_{\qc(Y)}(\qc(X), \qc(X))
$$
where the left hand side is equipped with convolution
and the right hand side with the composition of functors.
In other words, we have an identification of associative algebra objects
in the $\oo$-category
${\mc Pr}^{\rm L}$ of presentable $\oo$-categories with morphisms left adjoints.

In Section~\ref{convolution}, we calculate the center of $\qc(\XYX)$
in the following form. Recall that $\cL Y$ denotes the derived loop space of $Y$.

\begin{theorem}
Suppose $p:X\to Y$ is a map of perfect stacks that satisfies descent. Then there
is a canonical equivalence
$$\cZ(\qc(\XYX))\simeq\qc({\cL Y})$$
in which the central functor $$\qc({\cL Y})\simeq \cZ(\qc(\XYX))\to
\qc(\XYX)$$ is given by pullback and pushforward along the
correspondence $$\cL Y\longleftarrow \cL Y\times_Y X \longrightarrow
X\times_Y X.$$ 
\end{theorem}

When $p:X\to Y$ is also proper with invertible dualizing complex, 
we deduce an analogous statement identifying the trace of $\qc(\XYX)$ with
$\qc(\cL Y)$, conditional
on the validity of Grothendieck duality in the derived setting.

\begin{remark}
One can find a precursor to the above in the work of M\"uger \cite{Mueger} and Ostrik \cite{Ostrik}.

Given a semisimple abelian monoidal category
$\cC$ and a module category $M$, consider the monoidal
category 
$\cC_M^*$ 
consisting of $\cC$-linear
endofunctors of $M$. Then 
 independently of $M$, there is a 
canonical identification of the Drinfeld centers
 of $\cC_M^*$ and  $\cC.$
 
A motivating example is when $H\subset G$ are finite
groups, and one takes $\cC=\Rep(G)$ and
$M=\Rep(H)$, so that $ \cC^*_M \simeq \Vect(H\bs G/H)$. 
Then independently of $H$,
the center of  $\Vect(H\bs G/H)$ is the category of adjoint equivariant
vector bundles on $G$.

The above theorem extends this picture from finite groups to algebraic groups.
\end{remark}

\subsubsection{Example: affine Hecke categories}
As an illustration, we briefly mention a concrete application of the
above theorem to a fundamental object in geometric representation
theory.

Fix a reductive group $G$, and consider the Grothendieck-Springer
resolution $\tilde G\to G$ of pairs of a group element and a Borel
subgroup containing it.  Let $\cH_G^{\aff}$ be the quasi-coherent
affine Hecke $\infty$-category of $G$-equivarant quasi-coherent
sheaves on the Steinberg variety
$$
{\mathcal St}_G=\tilde{G}\times_{G}\tilde{G}.
$$
Work of Bezrukavnikov \cite{Roma ICM} and others
 places $\cH_G^{\aff}$ at the heart of many recent developments
in geometric representation theory.

Let us apply the above theorem with $X= \tilde G /G$ and $Y=G/G = \cL
BG$, where $G$ acts via conjugation.  Observe that the iterated loop
space $\cL Y=\cL (G/G) = \cL (\cL BG)$ is nothing more than the
derived moduli stack
$${\mc Loc}_G(T^2) =\Map(T^2, BG)
$$ of $G$-local systems on the torus.  Concretely, ${\mc Loc}_G(T^2)$
is the ``commuting variety" (or rather, commuting derived stack)
parameterizing pairs of commuting elements in $G$ up to simultaneous
conjugation.

\begin{cor} There is a canonical equivalence
$$\cZ(\cH_G^{\aff})\simeq \qc(\cL(\cL BG))\simeq \qc((BG^{S^1})^{S^1})\simeq \qc({\mc Loc}_G(T^2))$$
between the center of the quasi-coherent affine Hecke $\infty$-category and
the $\oo$-category of sheaves on the derived moduli stack of $G$-local systems on the torus.
\end{cor}

A similar
statement holds replacing the Grothendieck-Springer resolution $\tilde G\to G$
by the Springer resolution $T^*G/B\to {\mathcal N}$ of the
nilpotent cone. In this case, the center is equivalent to the $\oo$-category  of sheaves on the derived stack 
of pairs of a nilpotent and a
commuting group element up to simultaneous conjugation. 
In this linear version, one can also work
$\mathbb G_m$-equivariantly via the natural dilation action.

\begin{remark} 
We will not return to specific applications to representation theory in this paper,
but the interested reader will find further results along these lines in the paper \cite{character}
which studies integral
transforms in the context of $\D$-modules.  It includes applications to
the more familiar Hecke categories $\D(B\bs G/B)$ of $\D$-modules on
flag varieties. In particular, their Drinfeld centers are identified with
character sheaves on $G$, resulting in a Langlands duality for
 character sheaves.
\end{remark}

\subsection{Topological field theory}
As we discuss in Section~\ref{TFT}, our results may be viewed from the
perspective of topological field theory.  

The $\oo$-category $ \qc(\cL X)$ (the Drinfeld center $\cZ(\qc(X))$)
carries a rich collection of operations generalizing the braided
tensor structure on modules for the classical Drinfeld double.
Namely, for any cobordism $\Sigma$ between disjoint unions of circles
$(S^1)^{\coprod m}$ and $(S^1)^{\coprod n}$, we obtain restriction
maps between the corresponding mapping stacks
$$(\cL X)^{\times m}\leftarrow X^\Sigma \rightarrow (\cL X)^{\times
  n}.$$ Pullback and pushforward along this correspondence defines a
functor
$$
\qc(\cL X)^{\ot m}\to \qc(\cL X)^{\ot n}
$$ which is compatible with composition of cobordisms.  In particular,
we obtain a map from the configuration space of $m$ small disks in the
standard disk to functors
$$
\qc(\cL X)^{\ot m}\to \qc(\cL X).
$$

This equips $\qc(\cL X)$, and hence the Drinfeld center $\cZ(\qc(X))$,
with the structure of a (framed) $\cE_2$-category.  This establishes
the categorified (cyclic) Deligne conjecture in the current geometric
setting.

In particular for $X=BG$, $X^\Sigma$ is the derived moduli stack of
$G$-local systems on $\Sigma$ and we obtain a topological gauge
theory.

More generally, we have the following corollary of our main results:

\begin{cor}
Let $X$ be a perfect stack, and let $\Sigma$ be a finite simplicial
set. Then there is a canonical equivalence $\qc(X^\Sigma)\simeq
\qc(X)\otimes \Sigma$, where $X^{\Sigma}=\Map(\Sigma, X)$ denotes the
derived mapping stack, and $-\otimes \Sigma$ the tensor of stable
$\oo$-categories over simplicial sets.
\end{cor}

One can view this corollary as providing for a TFT over finite
simplicial sets.  We assign to such a simplicial set $U$ the
$\oo$-category $\qc(X^U)$, and for any
diagram of finite simplicial sets
$$ U\rightarrow \Sigma \leftarrow V,
$$ (for example a cobordism of manifolds) we obtain a correspondence
of mapping stacks
$$X^U \leftarrow X^\Sigma \rightarrow X^V,$$ 
and hence by pullback and pushforward a
functor 
$$\qc(X^U)\to \qc(X^V),
$$ satisfying composition laws with respect to gluings.

 In particular, since $\qc(X^{S^n})$
is equivalent to
the $\cE_n$-center $\cZ_{\cE_n}(\qc(X))$, we
obtain an action of the (framed) $\cE_{n+1}$-operad on
 $\cZ_{\cE_n}(\qc(X))$. This establishes the categorified (cyclic)
Kontsevich conjecture in the current geometric setting.

\subsubsection{Extended TFT}
An exciting recent development (postdating the submission of this
paper) is Jacob Lurie's announced proof of the Cobordism Hypothesis
\cite{jacob TFT} which characterizes extended TFTs
in arbitrary dimensions. The results of this paper may be used to
prove that the monoidal $\oo$-category $\qc(X)$, for a perfect stack $X$, defines an extended two-dimensional TFT. We briefly summarize this here (see \cite{character} for more
details regarding an analogous extended two-dimensional TFT).

The extended two-dimensional TFT $\cZ_X$ associated to a perfect
stack $X$ is a symmetric monoidal functor
$$ \cZ_X: 2Bord\to 2Alg
$$ from 
the $(\oo,2)$-category of unoriented bordisms:
\begin{enumerate}
\item[$\bullet$] objects: $0$-manifolds,
\item[$\bullet$] 1-morphisms: $1$-dimensional bordisms between $0$-manifolds,
\item[$\bullet$] 2-morphisms: classiying spaces of $2$-dimensional bordisms between $1$-bordisms,
\end{enumerate}
to the Morita $(\infty,2)$-category of algebras $2Alg$: 
\begin{enumerate}
\item[$\bullet$] objects: algebra objects in stable presentable $\oo$-categories,
\item[$\bullet$] 1-morphisms: bimodule objects in stable presentable $\oo$-categories,
\item[$\bullet$] 2-morphisms: classifying spaces of morphisms of bimodules.
\end{enumerate}

Note that given $A\in 2Alg$, we can pass to the $(\oo,2)$-category of
modules $\Mod_A$, and bimodules correspond to functors between
$(\oo,2)$-categories of modules.  Thus if a field theory assigns $A$ to a point,
we can also think of it as assigning $\Mod_A$ to a point.

The field theory $\cZ_X$ assigns the following to closed $0$, $1$ and
$2$-manifolds:
\begin{enumerate}
\item[$\bullet$] To a point, $\cZ_X$ assigns the monoidal
  $\oo$-category $\qc(X)$, or equivalently, the $(\oo,2)$-category of
  $\qc(X)$-linear $\oo$-categories.

\item[$\bullet$] To a circle, $\cZ_X$ assigns the Hochschild homology category
  of $\qc(X)$, which by our results can be identified with $\qc(\cL X)$.

\item[$\bullet$] To a closed surface $\Sigma$, $\cZ_X$ assigns the
  $k$-module of derived global sections $\Gamma(X^\Sigma, \cO_{X^\Sigma})$
  of the structure sheaf of the mapping space $X^\Sigma$.
\end{enumerate}

The proof that $\qc(X)$ satisfies the dualizability conditions of \cite{jacob TFT}
to define an extended TFT follows closely from the identification of
the Hochschild homology and cohomology categories of $\qc(X)$ (which is a
necessary consequence of the TFT structure). One could consult~\cite{character} for
more details in an analogous setting.


\subsubsection{Rozansky-Witten theory}
The topological field theory $\cZ_X$ introduced above is closely
related to well-known three-dimensional topological field theories.
When the target $X$ is the classifying stack $BG$ of a finite group,
$\cZ_X$ is the untwisted version of Dijkgraaf-Witten theory. More
generally, if $X$ is a 2-gerbe ($BB{\mathbb G}_m$-torsor) over $BG$
(classified by a class in $H^3(G,{\mathbb G}_m)$), we obtain the
twisted version (as studied in \cite{Freed}).

When $X$ is a smooth complex projective variety, $\cZ_X$ is closely
related to the $\ZZ/2$-graded three-dimensional topological field
theory associated to the holomorphic symplectic manifold $T^*X$ by
Rozansky-Witten \cite{RozW}, Kontsevich \cite{Kont RW} and Kapranov
\cite{Kap RW}. In particular see \cite{RW} and \cite{KRS} for work on
Rozansky-Witten theory as an extended topological field theory.

We confine ourselves here to a brief comparison of the two theories on the
circle. On the one hand, $\cZ_X$ assigns to $S^1$ the stable $\oo$-category
$\qc(\cL X)$.  Since $X$ is a smooth scheme, $\cL X$ is the total
space $T_X[-1] =\Spec \on{Sym}^\bullet \Omega_X[1]$ of the shifted tangent
bundle of $X$.  Thus we can identify $\qc(\cL X)$ with module objects in the
$\oo$-category $\qc(X)$ for the commutative algebra object $\on{Sym}^\bullet
\Omega_X[1]$. Via Koszul duality, this category is closely related to
modules for $\on{Sym}^\bullet T_X[-2]$, or in other words, sheaves on $T^*X$ with an
unusual grading. On the other hand, Rozansky-Witten theory assigns to
$S^1$ the $\oo$-category $\Perf(T^*X)$ of perfect complexes.  It would
be very interesting to develop Rozansky-Witten theory as a fully
extended TFT using the results of \cite{jacob TFT}, and explore its
relation with the derived algebraic geometry of $\qc(X)$.


\subsection{Acknowledgements} This paper relies heavily on the work of Jacob Lurie, whom the authors thank for many
helpful discussions and for his generosity in sharing his knowledge
(and in particular for the suggestion of how to describe the higher
$\cE_n$-centers of $\qc(X)$ geometrically).  DBZ and DN would like to
thank Bertrand To\"en for educating them in derived algebraic geometry
and patiently explaining many arguments.  They would also like to
thank Victor Ostrik for sharing his understanding of Drinfeld centers
and their applications to Hecke categories. JF thanks his advisor,
Michael Hopkins, for discussions on the topological field theory
aspects of this paper, as well as for his support and guidance. Many
thanks to Andrew Blumberg and Amnon Neeman for detailed comments on
early drafts. Finally, we would also like to thank the anonymous referees for their valuable feedback.

DBZ is partially supported by NSF CAREER grant
DMS-0449830.  JF was supported by an NSF Graduate Research
Fellowship. DN is partially supported by NSF grant DMS-0600909 and a
Sloan Research Fellowship. Work on this paper has taken place at the
Aspen Center for Physics and the IAS (supported by NSF grant
DMS-0635607), whom the authors thank for providing support and
stimulating environments.



\section{Preliminaries}\label{prelim}

In this section, we summarize relevant technical foundations 
in the theory of $\oo$-categories and derived algebraic geometry.
It is impossible to describe this entire edifice in a few pages,
so we will aim to highlight the particular concepts, results and their references
that will play a role in what follows.


\subsection{$\infty$-categories}
We are interested in studying algebraic operations on categories of a
homotopy-theoretic nature such as derived categories of sheaves or spectra. It
is well established that the theory of triangulated categories,
through which derived categories are usually viewed, is inadequate
to handle many basic algebraic and geometric operations. Examples include 
the absence of a good theory of gluing or of descent, of functor categories, 
or of generators and relations.
The essential problem is that 
passing to homotopy categories discards essential information (in particular,
homotopy coherent structures, homotopy limits and homotopy colimits).

This information can be captured in many alternative ways, the most common 
of which is the theory of model categories. Model structures keep weakly equivalent objects distinct but retain the extra
structure of resolutions which enables the formulation
of homotopy coherence. This extra structure can be very useful for calculations but makes 
some functorial operations difficult. In particular, it can be hard to 
construct certain derived functors because the given resolutions are inadequate.
There are also fundamental difficulties with the consideration of functor categories between model categories. 
However, much of the essential information encoded in model categories 
can be captured by the Dwyer-Kan simplicial localization. This construction uses the weak equivalences 
to construct {\em simplicial sets} (or alternatively, topological spaces) of maps between objects,
refining the sets of morphisms in the underlying homotopy category (which are recovered by passing to
$\pi_0$).

This intermediate regime between model categories and
homotopy categories is encoded by the theory of $(\infty,1)$-{\em categories}, or simply
$\infty$-categories. 
The notion of $\infty$-category captures (roughly speaking) the notion of a category
whose morphisms form topological spaces and whose compositions and
associativity properties are defined up to coherent homotopies. 
Thus an important distinction between $\infty$-categories and model categories or homotopy categories is that
coherent homotopies are naturally built in to all
the definitions. Thus for example all functors are naturally derived and 
the natural notions of limits and colimits in the $\infty$-categorical
context correspond to {\em homotopy} limits and colimits in more traditional formulations.

The theory of $\infty$-categories has many alternative formulations (as
topological categories, Segal categories, quasi-categories, etc; see
\cite{Bergner} for a comparison between the different
versions). We will follow the conventions
of \cite{topos}, which is based on Joyal's quasi-categories \cite{Joyal}. Namely, an $\infty$-category 
is a simplicial set, satisfying a weak version of the Kan condition guaranteeing the fillability of certain
horns. The underlying simplicial set plays the role of the set of objects while the fillable horns correspond
to sequences of composable morphisms. The book \cite{topos} presents a detailed study of $\infty$-categories,
developing analogues of many of the common notions of category theory (an overview of the $\infty$-categorical language,
including limits and colimits, appears in \cite[Chapter 1.2]{topos}). 

Among the structures we will depend on are the $\infty$-category of $\infty$-categories  \cite[3]{topos},
adjoint functors \cite[5.2]{topos}, and ind-categories and compact objects \cite[5.3]{topos} (see also Section \ref{air section}). 
Most of the objects we encounter form {\em presentable} $\infty$-categories \cite[5.5]{topos}.
Presentable $\infty$-categories are $\infty$-categories
which are closed under all (small) colimits (as well as limits, \cite[Proposition 5.5.2.4]{topos}), 
and moreover are generated in a weak sense
by a small category. In particular, by a result of Simpson \cite[Theorem 5.5.1.1]{topos}, they are
given by suitable localizations of $\infty$-categories of presheaves on a small $\infty$-category.
Presentable $\infty$-categories form an $\infty$-category ${\mc Pr}^{\rm L}$ 
whose morphisms are {\em continuous} functors,
that is, functors that preserve all colimits \cite[5.5.3]{topos}.
Note that since presentable categories are closed under all (co)products, categories with a finiteness condition
(like compact spaces, coherent sheaves, etc.) do not fall under this rubric.
A typical example is the $\infty$-category of spaces.

Algebra (and algebraic geometry, see below) in 
the $\infty$-categorical (or derived) setting has been developed in recent years by
To\"en-Vezzosi \cite{HAG1,HAG2} and Lurie \cite{dag1,dag2,dag3}. This has resulted in a very powerful and readily applicable formalism, 
complete with $\infty$-analogues of many of the common tools of ordinary category theory.
We single out two powerful tools that are crucial for this paper and available in the $\infty$-context thanks to \cite{dag2}
(but are not available in a suitable form in the triangulated or model contexts):
\begin{enumerate}

\item[$\bullet$] The general theory of descent, as embodied by the $\infty$-categorical version of the Barr-Beck theorem 
(\cite[Theorem 3.4.5]{dag2}).
\item[$\bullet$] The theory of tensor products of and functor categories between 
presentable $\infty$-categories (reviewed in Section \ref{monoidal infinity}).
\end{enumerate}

\subsubsection{Enhancing triangulated categories}
The $\infty$-categorical analogue of the additive setting of homological algebra
is the setting of stable $\infty$-categories \cite{dag1}. 
A stable $\infty$-category can be defined as an $\infty$-category with a zero-object, closed
under finite limits and colimits, and in which pushouts and pullbacks coincide 
\cite[2,4]{dag1}. The result of \cite[3]{dag1} is that stable categories are enhanced versions
of triangulated categories, in the sense that the homotopy category of a stable
$\infty$-category has the canonical structure of a triangulated category.
We will mostly be concerned with $\infty$-categories that are both presentable and stable,
as studied in \cite[17]{dag1}. Typical examples are the $\infty$-categorical enhancements of the derived
categories of modules over a ring, quasi-coherent sheaves on a scheme, 
and the $\infty$-category of spectra. 

Given a triangulated category which is linear over a ring $k$, we may
consider enhancing its structure in three different ways, promoting it to
\begin{enumerate}
\item[$\bullet$] a differential graded (dg) category,
\item[$\bullet$] an $A_{\infty}$-category, or 
\item[$\bullet$] a stable $\infty$-category.
\end{enumerate}
Among the many excellent references for dg and $A_\infty$-categories,
we recommend the survey \cite{Keller}. Relative a ring of
characteristic zero $k$, all three formalisms become equivalent:
$k$-linear stable $\infty$-categories are equivalent to $k$-linear
pre-triangulated dg categories (that is, those whose homotopy category
is triangulated).  Thus we recommend the reader interested in
characteristic zero applications substitute the term
``pre-triangulated $k$-linear dg category'' for ``stable
$\infty$-category'' throughout the present paper. The distinction
between $k$-linear stable $\oo$-categories, dg- and
$A_\infty$-categories becomes important when considering commutative
algebra away from characteristic zero, for which purpose we will only
consider the former.


\subsection{Monoidal $\infty$-categories}\label{monoidal infinity}
The definitions and results of this paper depend in an essential way
on noncommutative and commutative algebra for $\infty$-categories, 
as developed by Lurie in \cite{dag2} and \cite{dag3}.
We briefly summarize this theory in this section, giving 
detailed references for the benefit of the reader. 

The definition of a monoidal $\infty$-category is given in \cite[1.1]{dag2}.
The homotopy
category of a monoidal $\infty$-category is an ordinary monoidal category.
The $\infty$-categorical notion incorporates not only the naive notion of multiplication and unit on
an $\infty$-category $\cC$ but also all of the higher coherences for associativity, which are packaged in the data
of a fibration over $\Delta^{\rm op}$ whose fiber over $[n]$ is $\cC^{\times n}$. (Alternatively, 
it  can be captured concretely by a bisimplicial set with compatibilities, or as a monoid object 
in $\cC$ \cite[Remark 1.2.15]{dag2}, that is,
 a simplicial object in $\cC$ mimicking the classifying space of a monoid.) 
An algebra object in $\cC$ can then be defined \cite[1.1.14]{dag2} as an appropriate section of this fibration.
The $\infty$-categorical notion
of algebra reduces to the more familiar notions of $A_\infty$ $k$-algebra or $A_\infty$-ring spectrum
when the ambient monoidal $\infty$-category is that of differential graded $k$-modules or that of spectra
\cite[4.3]{dag2}. In other words, it encodes a multiplication associative up to coherent homotopies. 
Likewise, (left) modules over an algebra are defined by the same simplicial diagrams as algebras,
except with an additional marked vertex at which we place the module \cite[2.1]{dag2}. 
There is also a pairing to $\cC$ between left and right modules over an algebra object $A$ in $\cC$, namely, the relative
tensor product $\cdot \ot_A \cdot$ defined by the two-sided bar construction \cite[4.5]{dag2}.
Monoidal $\infty$-categories, algebra objects in a monoidal category, and module objects over an
algebra object themselves form $\infty$-categories, some of whose properties (in particular behavior
of limits and colimits) are worked out in \cite[1,2]{dag2}. In particular, limits of algebra objects
are calculated on the underlying objects \cite[1.5]{dag2} and module categories are stable \cite[Proposition 4.4.3]{dag2}.

The definition of a symmetric monoidal $\infty$-category is given in \cite[1]{dag3},
modeled on the Segal machine for infinite loop spaces. Namely, we replace $\Delta^{\rm op}$ 
in the definition of monoidal $\infty$-categories by the category of pointed finite sets,
thus encoding all the higher compatibilities of commutativity. Likewise, commutative algebra
objects are defined as suitable sections of the defining fibration. On the level of homotopy categories,
we recover the notion of commutative algebra object in a symmetric monoidal category,
but on the level of chain complexes or spectra this notion generalizes the notion of $\cE_\infty$-algebra
or $\cE_\infty$-ring spectrum (as developed in \cite{EKMM,HSS}). An important feature of the $\infty$-category of commutative
algebra objects is that coproducts of commutative algebra objects are calculated by
the underlying monoidal structure \cite[Proposition 4.7]{dag3}.
Section \cite[5]{dag3} introduces commutative modules
over commutative algebra objects $A$, which are identified with both left and right modules
over the underlying algebras. 
The key feature of the $\infty$-category
of $A$-modules is that it has a canonical symmetric monoidal structure \cite[Proposition 5.7]{dag3}, 
extending the relative tensor product of modules. Moreover, commutative algebras
for this structure are simply commutative algebras over $A$ \cite[Proposition 5.9]{dag3}.

One of the key developments of \cite{dag2} is the $\infty$-categorical version
of tensor products of abelian categories \cite{Deligne}.
Namely, in \cite[4.1]{dag2} it is shown that the $\infty$-category ${\mc Pr}^{\rm L}$ of 
presentable $\infty$-categories has a natural monoidal structure. In this structure,
the tensor product $\cC\ot\cD$ of presentable $\cC,\cD$ is a recipient of a universal functor from 
the Cartesian product $\cC\times \cD$ which is ``bilinear" (commutes with colimits in each variable separately). 
Moreover, \cite[Proposition 6.18]{dag3} lifts this to a symmetric monoidal structure in which
the unit object is the $\infty$-category of spaces. 
This structure is in fact closed, in the sense that ${\mc Pr}^{\rm L}$ has an internal hom functor
compatible with the tensor structure, see \cite[Remark 5.5.3.9]{topos} and \cite[Remark 4.1.6]{dag2}. 
The internal hom assigns to presentable $\infty$-categories
$\cC$ and $\cD$ the $\infty$-category of colimit-preserving functors $\Fun^{\rm L}(\cC,\cD)$, which is
presentable by \cite[Proposition 5.5.3.8]{topos}.
In Section \ref{defining centers} below,
we use the monoidal structure on ${\mc Pr}^{\rm L}$ to define 
an analogue for $\infty$-categories of the Hochschild cohomology of algebras or topological Hochschild cohomology
of ring spectra.

The symmetric monoidal structure on the $\infty$-category $\mc Pr^{\rm L}$ of
presentable $\infty$-categories restricts
to one on the full $\infty$-subcategory of stable presentable $\infty$-categories (\cite[4.2]{dag2} and \cite[6.22]{dag3}).
The unit of the restricted monoidal structure is the stable category of spectra. In particular, this induces a symmetric
monoidal structure on spectra and exhibits presentable stable categories as tensored over
spectra.
Thus if $\cC$ is a symmetric monoidal stable $\infty$-category (that is, a stable commutative ring object
in stable $\infty$-categories), we may consider module categories over $\cC$. These modules themselves
will form a symmetric monoidal $\infty$-category under the operation $\ot_\cC$, which is characterized
by the two-sided bar construction with respect to $\cC$. In our applications, $\cC$ will be 
the $\infty$-category $\qc(Y)$ of quasi-coherent sheaves on a derived stack $Y$, and we will consider the tensor products of 
module categories of the form $\qc(X)$ for derived stacks $X\to Y$. 

In Section \ref{defining centers}, we use this general formalism to define the center (or Hochschild cohomology)
and universal trace (or Hochschild homology) for algebra objects in any symmetric monoidal $\infty$-category.
The case of spectra recovers topological Hochschild (co)homology, while the case of
presentable $\oo$-categories ${\mc Pr}^{\rm L}$ 
provides a derived generalization of the Drinfeld center and will be the focus of our applications in Section \ref{applications}. 

In Section \ref{tensors of cats}, we discuss 
basic properties of $\infty$-categories of modules,
and  the tensor product of {\em small} stable $\infty$-categories.


\subsection{Derived algebraic geometry}\label{derived AG}
Algebraic geometry provides a wealth of examples of symmetric
monoidal categories, as well as powerful tools to study such 
categories. To a scheme or stack, we may assign its
category of quasi-coherent sheaves with its commutative multiplication given
by tensor product. This construction generalizes the
category of $R$-modules for $R$ a commutative ring (the case of
$\on{Spec} R$), and the category of representations of an algebraic
group $G$ (the case of classifying stacks $BG$). Conversely, Tannakian formalism
often allows us to reverse this process and 
assign a stack to a symmetric monoidal category.

Schemes and stacks are also natural sources of 
symmetric monoidal $\infty$-categories, which refine the familiar 
symmetric monoidal derived categories of quasi-coherent sheaves.
For example, to any stack in characteristic zero, we can assign the differential graded enhancement of its
derived category, constructed by taking the differential graded category of complexes
with quasi-coherent cohomology and localizing the quasi-isomorphisms
(following \cite{Keller, Drinfeld, Toen dg}). 
In general, we can consider the stable  symmetric monoidal
$\infty$-category
of quasi-coherent sheaves on any stack whose homotopy category
is the familiar derived category.

Trying to geometrically describe algebraic operations on
$\infty$-categories of quasi-coherent sheaves quickly takes us from ordinary algebraic geometry to
an enhanced version which has been developed over the last few
years known as derived algebraic geometry~\cite{dag, dag1, dag2, dag3, HAG1, HAG2}. In fact, derived
algebraic geometry also provides a far greater abundance of examples
of  stable symmetric monoidal $\oo$-categories. Thus it is most natural to both ask
and answer algebraic questions about  stable symmetric monoidal $\oo$-categories
 in this context.

Derived algebraic geometry generalizes the world of schemes
simultaneously in two directions. Regarding schemes in terms of
their functors of points, which are functors from rings to sets
(satisfying a sheaf axiom with respect to a Grothendieck topology),
we want to replace both the source and the target categories by
suitable $\infty$-categories. First, we may consider functors from
rings to spaces (or equivalently, simplicial sets) considered as an
$\infty$-category (with weak homotopy equivalences inverted). If we
consider only $1$-truncated spaces, or equivalently their
fundamental groupoids, we recover the theory of stacks. This
naturally leads to the introduction of higher stacks where we
consider sheaves of (not necessarily $1$-truncated) spaces on the
category of rings. For example, we can take any space
and consider it as a higher stack by taking (the sheafification
of) the corresponding constant functor on rings.

To pass from higher stacks to derived stacks, we replace the source
category of commutative rings by an $\infty$-category of commutative ring 
objects in a symmetric monoidal $\infty$-category.
There are at least three natural candidates that are commonly considered:

\begin{enumerate}
\item connective (that is, homological or non-positively graded)
commutative differential graded $k$-algebras over a commutative ring  $k$  of
characteristic zero;
\item simplicial commutative rings, or simplicial commutative $k$-algebras over a
commutative ring~$k$;
\item connective $\cE_\infty$-ring spectra, or connective $\cE_\infty$-algebras over
the Eilenberg-MacLane spectrum $H(k)$ of a commutative ring $k$.\footnote{Note 
that by the results of \cite{dag3} reviewed above, an $\cE_\infty$-algebra over $H(k)$ is equivalently
a commutative algebra object in the symmetric monoidal $\infty$-category
of $H(k)$-module spectra.}
\end{enumerate}

When $k$ is a $\QQ$-algebra, the $\oo$-categories of connective differential graded
$k$-algebras, simplicial commutative $k$-algebras, and connective
$\cE_\infty$-algebras over $H(k)$ are all equivalent. For a general commutative ring $k$,
simplicial commutative rings provide a setting for importing notions of homotopy theory into algebraic geometry over $k$ (for example, for the aim of describing centers of $\oo$-categories of sheaves on
schemes or stacks). Connective $\cE_\oo$-ring spectra are more subtle and provide the
setting for importing algebro-geometric notions back into stable homotopy theory. General
$\cE_\infty$-ring spectra provide a radical generalization which is at the heart of stable homotopy
theory. (See \cite{shipley} and references therein for the relation of algebra over rings and over the corresponding Eilernberg-MacLane spectra.)

The techniques and results of this paper apply equally in
any of the three settings, and we will refer to any of the three as
{\em commutative derived rings} without further comment.

Roughly speaking, a derived stack is a functor
from the $\oo$-category of commutative derived rings to the $\oo$-category
of topological spaces. It should satisfy a sheaf property with respect to a Grothendieck
topology on commutative derived rings.
Examples of derived stacks include:

\begin{enumerate}
\item spaces (constant functors),
\item ordinary schemes and stacks,
\item the spectrum of a derived commutative ring (the corresponding representable functor),
\item objects obtained by various gluings or quotients (colimits) and intersections or fiber products (limits)
of the above examples.
\end{enumerate}

We now recall the precise definitions of derived stacks, following~\cite{HAG1,HAG2, Toen}. 

The opposite of the $\oo$-category $\Alg_k$ of
derived commutative $k$-algebras admits a Grothendieck topology with respect to \'etale
morphisms. For $A, B\in \Alg_k$, a morphism $A\to B$ is \'etale if 
the induced morphism on connected component $\pi_0(A) \to \pi_0(B)$ is \'etale,
and for $i>0$, the induced map on higher homotopy groups is an isomorphism
$$
\xymatrix{
\pi_i(A) \ot_{\pi_0(A)} \pi_0(B)\ar[r]^-{\sim} & \pi_i(B).
}
$$

A finite family of morphisms $\{f_i : A \to  B_i\}$ is a \'etale covering if each $f_i$ is \'etale and
the induced morphism is surjective
$$
\xymatrix{
\coprod \Spec \pi_0(B_i)\ar@{->>}[r] &  \Spec \pi_0(A).
}
$$
This induces a Grothendieck topology on the (opposite of the) homotopy category of $\Alg_k$.

Now a derived stack $X$ is a covariant functor from $\Alg_k$ 
to the $\oo$-category $\Top$ of topological spaces which is a sheaf with respect to the \'etale topology. In particular, for any \'etale cover $A\to B$, the induced morphism
$$
\xymatrix{
X(A) \ar[r]^-{\sim} & \lim X(B^*)
}$$is an equivalence, where $B^*$ is the standard cosimplicial resolution of $A$ with $i$th simplex $B^{\ot_A i+1}$.

We will work exclusively with derived stacks whose diagonal morphism $\Delta:X\to X\times X$ is
representable and affine.
Given a quasi-compact derived stack $X$ with affine diagonal, we can choose a cover $U\to X$ by an
affine derived scheme, and obtain a \u Cech simplicial affine
derived scheme $U_*\to X$ with $k$-simplices given by the $k$-fold
fiber product $U\times_X\cdots\times_X U$ and whose geometric
realization is equivalent to $X$.


\subsection{Derived loop spaces}
Even if we are interested in studying primarily schemes,
the world of derived stacks is a necessary setting in which to calculate homotopically correct
quotients, fiber products and mapping spaces. 
We illustrate this with an important geometric operation on derived stacks: the formation
of the derived loop space.  This operation is one of our motivations for
considering derived stacks in the first place, since the derived loop 
space of an ordinary scheme or stack is already a nontrivial derived stack.
(See \cite{cyclic} for a different appearance of derived loop spaces in relation to cyclic homology, $\D$-modules and
representation theory.) 

The free loop space of a derived stack $X$ is the internal hom
$\cL X=X^{S^1}=\Map(S^1,X)$
of maps from the
constant stack given by the circle $S^1$. As a derived stack, the loop space may be
described explicitly as the collection of pairs of points in $X$
with two paths between them, or in other words, as the derived
self-intersection of the diagonal
$$\cL X \simeq X\times_{X\times X} X.$$

Let us illustrate this notion in a few examples. 

For $X$ a
topological space (constant stack), $\cL X$ is of course the free
loop space of $X$ (again considered as a constant stack). 

For $X=BG$
the classifying space of an algebraic group, $\cL X=G/G$ is the
adjoint quotient of $G$, or in other words, the adjoint group for
the universal bundle $EG={\rm pt}\to BG$. Note that this agrees with the
underived inertia stack of $BG$. 

For $X$
a smooth variety over a field of characteristic zero, the derived self-intersection of the diagonal can
be calculated by a Koszul complex to be the spectrum of the complex
of differential forms on $X$ (with zero differential) placed in
homological degrees. 
We thus obtain that the
loop space of a smooth variety $X$ can be identified with the
relative spectrum of the symmetric algebra of the shifted tangent bundle
$T_X[-1]$. 

Similarly, for any $n\geq 0$, we may consider the derived stack of
maps $X^{S^n}=\Map(S^n, X)$ from the $n$-sphere $S^n$ into $X$. Concretely, the
presentation of $S^n$ by two cells glued along $S^{n-1}$ leads to an iterated
description of $X^{S^n}$ as the self-intersection
$$X^{S^n}=X
\times_{X^{S^{n-1}}} X
$$
along two copies of constant maps over $X^{S^{n-1}}$.

More generally, for any topological space $\Sigma$ and stack $X$ we can consider the
derived mapping
stack $X^\Sigma =\Map(\Sigma, X)$. 
 As we will demonstrate,
 this construction suggests the world of derived algebraic geometry
is a natural setting to construct topological $\sigma$-models.


\subsection{$\cE_n$-structures}\label{monoidal structures}
When considering monoidal structures on vector spaces, we have the
option of considering associative or commutative multiplications.
When considering monoidal structures on categories, we are faced
with three levels of increasing commutativity: ``plain" monoidal
categories carrying an associative multiplication; braided monoidal
categories, in which there is a functorial isomorphism exchanging
the order of multiplication $A\otimes B\to B\otimes A$ and
satisfying the braid relations; and symmetric monoidal categories,
in which the square of the braiding is the identity. 

In homotopy theory, there is an infinite sequence of types of
algebraic structures interpolating between associativity and
commutativity, modeled on the increasing commutativity of $n$-fold
iterated loop spaces. These algebraic structures can be encoded by the
little $n$-disk or $\cE_n$ operad for $n\geq 1$ or $n=\infty$, which
is an operad in the category of topological spaces.  Recall that an
operad $F$ (in spaces) is a sequence of spaces $F(k)$, which
parametrize $k$-fold multiplication operations in the algebraic
structures we are encoding, together with actions of the symmetric
group permuting the entries and equivariant composition maps. The
$k$-th space of the $\cE_n$ operad parametrizes disjoint collections
of $k$ small balls in the $n$-ball, with the natural
``picture-in-picture" composition maps.  An $\cE_n$-vector space
carries operations labelled by components of the $\cE_n$ operad, and
is an associative algebra for $n=1$ and a commutative algebra for
$n>1$. The notion of $\cE_n$-category for a (usual discrete) category
is sensitive to the fundamental groupoid of the $\cE_n$ operad,
leading to the three different notions of monoidal category ($n=1$),
braided monoidal category ($n=2$) and symmetric monoidal category
($n>2$, since for $n >2$ the spaces in the $\cE_n$ operad are simply
connected).  However, even on the level of graded vector spaces, with
operations labelled by the homology of the $\cE_n$ operad, we obtain
different notions for every $n$, with $n=2$ giving the notion of a
Gerstenhaber algebra familiar from the study of Hochschild cohomology.

We have already encountered the notions of algebra object
(corresponding to the case $\cE_1$, or equivalently $A_\infty$) and
commutative algebra object (corresponding to $\cE_\infty$) in the
context of a symmetric monoidal $\infty$-category, such as dg modules
over a ring $R$, spectra or presentable $\infty$-categories.  In
\cite{thesis}, the general theory of algebras over operads in
$\infty$-categories is developed and applied to algebra and geometry
in the $\cE_n$ case. Roughly speaking, in the $\infty$-categorical
context, we consider the operadic operations and compositions in a
homotopy coherent fashion (see Section~\ref{sect higher centers} for
more details).  Thus for example, an $\cE_2$-category is a
homotopy-theoretic analogue of a braided monoidal category.

The notions of $\cE_n$-algebras and categories pervade homotopy
theory, but have also become prominent in algebra and topological
field theory. We briefly mention two such applications.

Deligne's Hochschild cohomology conjecture asserts that the Hochschild
cochain complex of an associative algebra is an $\cE_2$-algebra,
lifting the Gerstenhaber algebra structure on Hochschild cohomology.
Kontsevich's conjecture generalizes this to assert that the
Hochschild cohomology of an $\cE_n$-algebra is an $\cE_{n+1}$-algebra.

The space of states associated to an $n-1$-sphere by an
$n$-dimensional topological field theory has a natural $\cE_n$-structure, given by tree-level field theory operations, independent of
where the field theory takes its values (vector spaces, chain
complexes, categories, etc.)

\begin{remark}
After the completion of this paper, the paper \cite{dag3} was revised
to include a thorough treatment of $\oo$-categorical operads and their
algebras. Furthermore, the paper \cite{dag6} studies in great detail the
specific case of the $\cE_n$-operads (including a proof of general
versions of the Deligne-Kontsevich conjecture). We refer the reader to
these preprints for details on these topics.
\end{remark}



\section{Perfect Stacks}\label{perfect}

By an $\oo$-category, we will always mean an $(\infty,1)$-category without
further comment, and refer to Section \ref{prelim} for an overview of the required
aspects of the general theory. 
For the reader accustomed to working with model categories, it is important to note
that in an $\oo$-category, all tensors, homs, limits, colimits, and other usual operations  are taken
in a derived or homotopical sense. In other words,
coherent homotopies are automatically built in to all definitions, 
though it may not explicitly appear in the notation. For example, a colimit in
the $\oo$-category obtained as the localization of a model category corresponds
to a homotopy colimit in the original model category.

\medskip

Throughout the rest of the paper, we fix a derived commutative ring
$k$ (in any of the three senses discussed in Section \ref{derived AG})
and work relative to $k$. We use the phrase derived commutative
$k$-algebra to refer to a commutative algebra object in $k$-modules.

For $k$ a commutative $\QQ$-algebra, one can replace $k$-modules with
chain complexes over $k$, derived commutative $k$-algebras with
commutative differential graded $k$-algebras, and stable $k$-linear
$\oo$-categories with pre-triangulated $k$-linear differential graded
categories.

\medskip

%

\medskip


In Section \ref{air section}, we review various notions of when
an $\oo$-category is finitely generated, and introduce the class of perfect stacks to
which our results apply. In Section \ref{base change}, we
establish base change and the projection formula for perfect morphisms. 
In Section
\ref{sect properties of air}, we show that many common classes of
stacks are perfect.


\subsection{Definition of perfect stacks} \label{air section}

Our main objects of study are $\infty$-categories of quasi-coherent
sheaves on derived stacks. For the foundations of derived stacks and
quasi-coherent sheaves on them, we refer the reader
to~\cite{HAG1,HAG2, Toen}. (Note that a theory of descent for sheaves
on higher stacks was developed previously by Hirschowitz and Simpson
\cite{hirschowitzsimpson}.)

Given a derived stack $X$, we have the stable symmetric monoidal
$\oo$-category $\qc(X)$ of quasi-coherent sheaves on $X$. To recall
its construction, consider first a derived commutative $k$-algebra
$A$, and the representable affine derived scheme $X=\Spec A$. In this
case, one defines $\qc(X)$ to be the $\infty$-category of $A$-modules
$\Mod_A$ (i.e., module objects over $A$ in $k$-modules). Its homotopy
category is the unbounded derived category of $A$.

In general, any derived stack $X$ can be written as a
colimit of a diagram of affine derived schemes $X\simeq \colim_{U\in\Aff_{/X}}U$. Then
one defines $\qc(X)$ to be the limit (in the $\oo$-category of $\oo$-categories)
of the corresponding diagram
of $\infty$-categories $$\qc(X):= \lim_{U\in\Aff{/X}}\qc(U).
$$ 
One can think of an object
$F\in\qc(X)$ as collections of quasi-coherent sheaves $F|_U$ on the
terms $U$ together with compatible identifications between their
pullbacks under the diagram maps.

When $X$ is quasi-compact and has affine diagonal, by choosing an
affine cover $U\to X$ with induced \u Cech simplicial affine derived
scheme $U_*\to X$, we can realize $\qc(X)$ by a smaller limit, the
totalization of the cosimplicial diagram~$\qc(U_*)$.

\medskip

An important feature of the $\infty$-category $\qc(X)$ is that it is
cocomplete, that is, closed under all small colimits (or equivalently,
since $\qc(X)$ is stable, all small coproducts). Nevertheless, it can
be difficult to control $\qc(X)$ algebraically via reasonable
generators. In general, it is convenient (and sometimes indispensable)
to work with $\infty$-categories that are ``generated by finite
objects" in a suitable sense. Let us summarize well known approaches
to this idea.  In a moment, we will provide a more detailed
discussion.

There are two common notions of when
a small $\infty$-subcategory $\cC^{\circ}$
generates an $\infty$-category $\cC$.
On the one hand, we could ask that $\cC$ be the inductive limit  $\Ind \cC^\circ$.
On the 
other hand, we could ask
that in $\cC$  the 
right orthogonal of $\cC^{\circ}$ vanishes.

There are three common notions of when an object should be considered finite:
 perfect objects, dualizable objects, and compact objects, which refer
respectively to the geometry, monoidal structure, and categorical structure of $\qc(X)$.

We now introduce the class of perfect stacks. We will check below that for perfect stacks,
the above notions of generators and finite objects all coincide.

\begin{definition}\label{perfect def}
Let $A$ be a derived commutative ring. An $A$-module $M$ is perfect if lies in the smallest $\oo$-subcategory of $\Mod_A$ containing $A$ and closed under finite colimits and retracts. For a derived stack $X$, $\oo$-category $\Perf(X)$ is the full $\oo$-subcategory of $\qc(X)$ consisting of those sheaves $M$ whose restriction $f^* M$ to any affine $f: U\ra X$ over $X$ is a perfect module. 
\end{definition}

\begin{definition}
A derived stack $X$ is said to be perfect if
it has affine diagonal and the $\infty$-category $\qc(X)$ is the inductive limit
$$
\qc(X) \simeq \Ind\Perf(X)
$$
of the full $\infty$-subcategory $\Perf(X)$ of perfect complexes.

A morphism $X\to Y$
is said to be perfect if its fibers $X\ti_Y U$ over affines $U\to Y$ are perfect.
\end{definition}

See \cite[5.3.5]{topos} 
for the construction of Ind-categories of $\infty$-categories, and \cite[8]{dag1} where it is shown that
Ind-categories of stable 
$\infty$-categories are stable.
Let us mention that in the Ind-category $\Ind\cC$ of an $\oo$-category $\cC$,
morphisms between Ind-objects can be calculated via the expected formula
$$
\Hom_{\Ind\cC}
(\ind [X_i], 
(\ind [Y_j])
\simeq
\lim
\colim
\Hom_{\cC}(X_i, Y_j).
$$
For the reader unaccustomed to Ind-categories,
 we will momentarily give an alternative formulation of perfect stack
in the more familiar language of compactly generated categories.

We next proceed with a review of the various notions of generators and finite objects.
In Section~\ref{base change}, we show that perfect morphisms satisfy base
change and the projection formula. In Section \ref{sect
properties of air}, we show that the class of perfect stacks includes
many common examples of interest, and
 check that any morphism between perfect stacks is itself perfect.


\subsubsection{Finite objects}
We review here the three common notions of finite objects and their interrelations
(See \cite[17]{dag1} and \cite[4.7]{dag2}
for more details, as well as~\cite{BvdB, HPS, Keller, LMS} among many other sources).
We remind the reader that we are working in the context of $\oo$-categories,
so constructions such as colimits correspond to {homotopy} colimits in the context of model categories.

\begin{definition}\

\begin{enumerate}
\item An object $M$ of a stable $\oo$-category $\cC$ is said to be
{\em compact} if $\Hom_\cC(M,-)$ commutes with all coproducts
(equivalently, with all colimits).

\item An object $M$ of a stable symmetric monoidal $\oo$-category $\cC$ is said
to be (strongly) {\em dualizable} if there is an object $M^\vee$ and
unit and trace maps
$$
\xymatrix{ 1  \ar[r]^-{u} & M\ot M^\vee
 \ar[r]^-{\tau} & 1
}
$$ such that the composite map
$$
\xymatrix{ M  \ar[r]^-{u \ot \on{id}} & M\ot M^\vee\ot M
\ar[r]^-{\on{id}\ot\tau}  & M }
$$ 
is the identity.

\end{enumerate}
\end{definition}

Suppose $\cC$ is a stable presentable $\oo$-category (such as
$\qc(X)$).  Then an object $M\in \cC$ is compact if and only if maps
from $M$ to any small coproduct factor through a finite coproduct.
Furthermore, a functor $F:\cC\to \cD$ between stable presentable
$\oo$-categories that preserves finite colimits preserves small
colimits if and only if it preserves small coproducts.
(See~\cite[Proposition 17.1]{dag1}.)

In a closed symmetric monoidal 
$\infty$-category $\cC$ (such as $\qc(X)$),
an object $M\in \cC$
is dualizable if and only if there exists a coevaluation map
$$
1\to M\ot \intHom(M,1)
$$
satisfying the appropriate conditions (since one already has an
evaluation map). If an object $M\in \cC$ is dualizable, then we can
turn internal Hom from $M$ into tensor product with $M^\vee$ in the
sense that there is a canonical equivalence
$$
\intHom(M,-)\simeq M^\vee \ot (-).
$$
In particular, this implies that $\intHom(M,-)$ preserves colimits
and $M\ot-$ preserves limits:
$$M\otimes\lim N_\alpha \simeq \Hom_\cC( 1_\cC, M \otimes \lim N_\alpha)
\simeq \Hom_\cC(M^\vee, \lim N_\alpha)\simeq \lim \Hom_\cC(M^\vee,
N_\alpha) \simeq \lim M\otimes N_\alpha.
$$

It is enlightening to note the following characterization of
dualizable objects, which parallels the definition of compact objects
(but will not be used in this paper).

\begin{lemma} Let $\cC$ be a symmetric monoidal presentable stable
  $\oo$-category, whose monoidal structure distributes over
  colimits. An object $M$ of $\cC$ is then dualizable if and only if
  tensoring with $M$ preserves all limits.
\end{lemma}

\begin{proof}
  The necessity of $M\ot -$ preserving limits is noted above ($\cC$ is
  closed by virtue of being presentable with monoidal structure
  distributing over colimits, see \cite[Proposition 2.1.12]{dag2}). To
  demonstrate sufficiency, assume that $M\ot -$ preserves limits, and
  then consider the endofunctor of $\cC$ defined by tensoring with
  $M$.  By assumption on $M$ and $\cC$, this functor preserves all
  limits and colimits. We may now apply the adjoint functor theorem of
  \cite{topos} to deduce the existence of a left adjoint $F$ to $M\ot
  -$. Denote by $M^\vee$ the value $F(1_\cC)$ of $F$ applied to the
  unit of $\cC$. The existence of unit and trace maps $1_\cC \ra M\ot
  M^\vee \ra 1_\cC$ is now a particular instance of the unit and
  counit maps for this adjunction, which implies that $M$ and $M^\vee$
  are in duality. Hence $M$ is dualizable.

\end{proof}

In a general stable presentable symmetric monoidal $\oo$-category $\cC$,
the classes of compact and dualizable objects 
do not coincide. In particular, the monoidal unit $1\in \cC$ is always dualizable but
not necessarily compact. 

In the case of a derived stack $X$, the unit $\cO_X \in \qc(X)$ is
compact if and only if the global sections functor $\Gamma(X,-)$
preserves colimits. (This fails if the global sections $\Gamma(X,
\cO_X)$ are too big such as in the following examples: ind-schemes
such as the formal disk $\on{Spf} k[[t]]$; the classifying space of a
topological group such as $BS^1$; the classifying space of finite
groups in modular characteristics.) However, if the unit $\cO_X \in
\qc(X)$ is itself compact then all dualizable objects are compact,
since Hom from a dualizable object $M$ is the composition of the
colimit preserving functors internal Hom $\intHom(M,-)$ and global
sections $\Gamma(X,-)$.

\begin{lemma}[\cite{BoN}, 6.4, \cite{EKMM} III.7.9, \cite{dag2}
  4.7.2]\label{affine air}
For the $\infty$-category $\Mod_k=\qc(\Spec k)$ of modules over a commutative derived ring 
(that is, quasi-coherent sheaves on 
an affine derived scheme), all three notions of finiteness coincide: $M$ compact $\iff$ $M$
dualizable $\iff$  $M$ perfect.
\end{lemma}

\begin{proof} First, note that
the free module $k$, which is the monoidal unit, is clearly compact. Hence 
all dualizable objects are compact.
Moreover, we can write any
object as a colimit of free modules. For $M$ compact, the identity map $\Id_M\in\Hom(M,M)$ has to
factor through a {\em finite} colimit, showing that $M$ is perfect.
Finally, perfect modules are dualizable since we can explicitly exhibit
their dual as a finite limit of free modules. 
\end{proof}

It is useful to note that the notion of dualizable is {\em local}.  On
the one hand, pullback for any map of stacks (for example, restriction
to an affine) preserves dualizable objects.  On the other hand, a dual
object with its unit and trace maps is functorially characterized,
thus if it exists locally, it will glue together to a global object.
This observation leads to the identification of perfect and dualizable
objects in $\qc(X)$ for any $X$:

\begin{prop}\label{perfect=dualizable}
For a derived stack $X$, an object of $\qc(X)$ is dualizable if and
only if it is perfect.
\end{prop}

\begin{proof}
Let $M\in \qc(X)$ be dualizable with dual $M^\vee$. Then for any map $\eta: \Spec A \ra X$, 
the pullback $\eta^* M$ is dualizable with dual $\eta^* M^\vee$. Dualizable objects of $\Mod_A$ are perfect, hence $\eta^* M$ is perfect and so by definition, $M$ is perfect.

Now suppose $M\in \qc(X)$ is perfect. Recall that by definition, we have  
$$\qc(X) \simeq \lim_{\Spec A\in {\it Aff}/X}\Mod_A.
$$
Since $M$ is perfect, for any map $\eta:\Spec A\to X$, the pullback $\eta^* M$ is perfect, hence dualizable.
We take the value of the dual $M^\vee$ along a map $\eta:\Spec A\to X$ 
to be the dual of the pullback $(\eta^* M)^\vee$. Note that $M^\vee$ is well-defined, since for any composite $\eta\circ\nu: \Spec B \ra X$, there is a natural equivalence 
$(\nu^* \eta^* M)^\vee\simeq \nu^*((\eta^* M)^\vee)$.

To exhibit $M$ and $M^\vee$ as dual to one another, we must construct the requisite unit and counit maps $u: \cO_X \ra M\ot M^\vee$ and $c: M^\vee \ot M \ra \cO_X$. Again using the definition of $\qc(X)$ as a limit, to produce one of these maps, it suffices 
to define analogous maps for the pullbacks under each $\eta: \Spec A \ra X$
which themselves are compatible under pullbacks.
But the existence of such  maps are an immediate consequence of the definition $\eta^*M^\vee = (\eta^* M)^\vee$.
Finally, to verify that the usual composititions $M \ra M\ot M^\vee \ot M \ra M$ and $M^\vee \ra M^\vee \ot M \ot M^\vee \ra M^\vee$ are equivalences, it suffices to check under pullbacks to affines.
But this is a direct consequence of our definition of $M^\vee$ and the fact that pullbacks preserve tensor products. 
\end{proof}


\subsubsection{Generators}\label{generators}
Now we review notions of what it means for compact objects to generate a
stable $\oo$-category.
(See \cite[17]{dag1} for more details,
and \cite[5.5.7]{topos} for the general setting of presentable $\infty$-categories.)

\begin{definition}
A stable category $\cC$ is said to be {\em compactly generated} 
if there is a small $\infty$-category $\cC^\circ$ of compact objects $C_i\in
\cC$ whose right orthogonal vanishes: if $M\in\cC$ satisfies
$\on{\Hom}_\cC(C_i,M)\simeq 0$, for all $i$, then $M\simeq 0$.
\end{definition}

As explained in \cite[Remark 17.3]{dag1}, whether a stable $\oo$-category is compactly generated
 can be studied completely in the underlying homotopy category.
In particular, the notion for stable $\infty$-categories is compatible with that for triangulated categories.

\begin{example}
For a commutative derived ring $k$,
the stable $\oo$-category $\Mod_k$ of $k$-modules is
compactly generated. In fact, it is generated by the free module $k$ itself.
\end{example}

On the one hand, for a stable small $\infty$-category $\cC^{\circ}$, the inductive limit $\cC=\Ind(\cC^{\circ})$ is a compactly generated stable presentable $\infty$-category. 
Furthermore (see~\cite[5.3.4]{topos}), if $\cC^{\circ}$ is closed under finite colimits
and idempotent complete,
then it can be recovered as the compact objects of $\cC$.
In particular, we have that $\cC$ is the Ind-category of its compact objects $\cC^\circ$.

On the other hand,
given a stable $\infty$-category $\cC$
with a  small full $\infty$-subcategory $\cC^\circ$ of compact generators, one can recover
all compact objects of $\cC$ by a result of Neeman \cite{neemanTTY} (see also \cite[Proposition 5.3.4.17]{topos}):
the compact objects are precisely direct summands of the objects of
the smallest stable $\infty$-subcategory $\cC_s$ containing
$\cC^\circ$ (that is, they are direct summands of finite iterated extensions of
objects of $\cC^\circ$). 
In particular, if $\cC^\circ$ is stable and idempotent complete, then it consists
precisely of the compact objects of $\cC$.

If we further assume that $\cC$
is a presentable stable $\infty$-category $\cC$
with a  small full $\infty$-subcategory $\cC^\circ$ of compact generators, then a theorem of 
Schwede and Shipley \cite{schwedeshipley} 
guarantees that we can recover $\cC$ as the cocompletion of $\cC_s$ 
(see \cite[4.4]{dag2} for the $\infty$-categorical
version, and \cite{Keller} for the differential graded version).
In other words, we recover $\cC$ by passing to the category of
colimit preserving $k$-linear functors to $k$-modules
$$
\cC\simeq
\Fun(\cC_s^{\rm op},\Mod_k).
$$

In particular, we now can check that our notion of perfect stack is equivalent
to more familiar assumptions on a symmetric monoidal $\oo$-category.

\begin{prop}
For a derived stack $X$ with affine diagonal, the following are equivalent:
\begin{enumerate}
\item $X$ is perfect.
\item $\qc(X)$ is
compactly generated,
and its compact and dualizable objects coincide.
\end{enumerate}
\end{prop}

\begin{proof}
If $X$ is perfect, so that $\qc(X)=\Ind \Perf (X)$,
we claim that compact and dualizable objects agree, and hence compact objects generate, so that
(1) implies (2). 

To see the claim, it suffices to show that the full $\oo$-subcategory of dualizable 
objects of $\qc(X)$ is idempotent complete (since it is also stable). 
Since $\qc(X)$ is idempotent complete, this 
is equivalent to showing that dualizable objects 
are closed under retracts. However, for a retract $N$ of a dualizable object $M$ 
one can explicitly write  down the unit and trace maps for $N \ot \intHom(N, \cO_X)$ and 
confirm the necessary conditions. We leave this to the reader.

Conversely, if $\qc(X)$ is compactly generated, then $\qc(X)$ is the Ind-category of its compact
objects, and hence by assumption also the Ind-category of its dualizable objects, so that (2) implies (1). 
\end{proof}




\subsection{Base change and the projection formula}\label{base
change} In this section, we collect properties of the pushforward
along a perfect stack, as summarized in the following proposition:

\begin{prop}\label{Thomason pushforwards} Let $f: X\ra Y$ be a perfect morphism. 
Then $f_*:\qc(X)\to \qc(Y)$ commutes with all small
colimits and satisfies the projection formula. Furthermore, if $g:
Y' \ra Y$ is any map of derived stacks, the resulting base change
map $g^* f_* \ra f'_* g'^*$ is an equivalence.
\end{prop}

\begin{remark}
The conclusions of the proposition in fact do not require the full
strength of the assumption that $f$ be perfect. The only hypothesis
needed is that $f_*$ preserve colimits, or equivalently that the
unit is relatively compact.
\end{remark}

\begin{remark}
In the setting of simplicial commutative rings, and for $f:X\to Y$ a
bounded, separated, quasi-compact relative  derived algebraic space, the following
results can essentially be found in Proposition 5.5.5 of Lurie's
thesis. The arguments in this section are an expanded version of parts of the
arguments there, which extend largely unchanged to any reasonable
setting for derived algebraic geometry such as $\cE_\infty$-ring
spectra.
\end{remark}

Let us first consider the case when $Y=\Spec A$ is affine, so that
$f:X\to \Spec A$ is perfect over $A$. Then the pushforward $f_*$
coincides with the global sections functor $\Gamma:\qc(X)\ra
\Mod_A$. Over an affine base, $f_*$ is colimit preserving if and
only if the structure sheaf $\cO_X$ (which is always dualizable) is
a compact object of $\qc(X)$, which follows from $f$ being perfect.

\begin{lemma} Let $M\in \qc(X)$, and let $N\in \qc(Y) \simeq \Mod_A$.
Then the natural projection map $f_*M \otimes N \ra f_*(M\otimes f^*
N)$ is an equivalence.
\end{lemma}

\begin{proof} Tensor products and pullbacks always preserve
colimits, and in our setting $f_*$ is colimit preserving as well.
Therefore for any $M$ the functors $f_* M\otimes (-)$ and $f_*(M
\otimes f^*-)$ define colimit preserving endofunctors of $\Mod_A$.
Hence both functors are determined by their value on $A$, and
canonically take the value $f_* M$. We find that the natural map is
an equivalence.
\end{proof}

Let us continue with $f: X \ra \Spec A$ as above, and consider $g:
\Spec B \ra \Spec A$ an arbitrary map of affine derived schemes.
Consider the Cartesian square.
$$
\xymatrix{
X\times _A \Spec B\ar[d]^{g'} \ar[r]^-{f'}& \Spec B \ar[d]^g\\
X \ar[r]^-f & \Spec A\\}$$

\begin{lemma}
The natural base change morphism $g^* f_* \ra f'_* g'^*$ is an
equivalence.
\end{lemma}

\begin{proof}
Since $g$ is affine and hence $g'$ is as well, the fiber product
$X\times_A \Spec B$ can be identified with the relative spectrum
$\Spec_X C$ of a commutative algebra object $C \in {\rm
Alg}(\qc(X))$, and $g'_*$ induces an equivalence $\qc({X\times_A
\Spec B}) \simeq \Mod_C(\qc(X))$. Furthermore, the pullback $g'^*$
can be described as tensoring with $C$, and thus in particular $f'_*
g'^* M \simeq f_* (C\otimes M)$. However, the global sections
functor takes fiber products to tensor products, so we can identify
$C \simeq f^* g_* B$. Applying the previously established projection
formula twice, we can now compute $f_* (f^*g_* B \otimes M) \simeq
f_* M \otimes_A g_* B \simeq g^* f_* M \otimes_B B \simeq g^* f_*
M$, completing the proof.
\end{proof}

Now consider the general case where $f: X \ra Y$ is any perfect morphism. 
Let us define a pushforward functor $f_+:\qc(X)\ra\qc(Y)$
by {\it requiring} it to satisfy base change for affine derived
schemes over $Y$. That is, for any $M\in \qc(X)$, let us define $f_+
M$ to take the value $(f_+ M)(g) = g'_* f'^* M$, for any $g:\Spec
B\to Y$. As a corollary of the previous lemma, we can verify that
this definition is sensible.

\begin{cor} $f_+ M$ is a well-defined quasi-coherent sheaf on $Y$.
\end{cor}

\begin{proof}
Since $\qc(Y) \simeq \lim_{\Aff/Y} \Mod_B$, the claim that $f_+ M$
forms a quasi-coherent sheaf on $Y$ is equivalent to the claim that
for any diagram of the form
$$
\xymatrix{ X\times_Y \Spec C\ar[d]^-{f''}\ar[r]^-{h'}&X\times_Y
\Spec B
\ar[d]^{f'}\ar[r]^-{g'}& X\ar[d]^-f\\
\Spec C \ar[r]^-h&\Spec B\ar[r]^-g&Y\\
}
$$
$f_+M(g \circ h)$ is canonically equivalent to $h^*f_+ M(g)$.
Unraveling these formulas, by definition we have that $f_+M(g\circ
h) = f''_*h'^* g'^* M$ and that $h^*f_+ M(g) \simeq h^* f'_* g'^*
M$. By the previous lemma, these are equivalent by base change in
the left hand square: $f':X\times_Y \Spec B\ra \Spec B$ is perfect, and
so $h^* f'_* \simeq f''_* h'^*$.
\end{proof}

\begin{lemma}\label{qc and fiber} The natural transformation $f_* \ra f_+$ is an equivalence
of functors.
\end{lemma}
\begin{proof} Take any $N\in\qc(Y)$ and $M\in \qc(X)$. First note that for
any quasi-coherent sheaf $K \in\qc(Z)$, there is an equivalence
$K\simeq \lim_{g\in \Aff/Z} g_* g^* K$. Thus we calculate
$$
f_+M \simeq \lim_{\Aff/Y} g_* g^* f_+ M \simeq \lim_{\Aff/Y} g_*
f'_* g'^* M \simeq \lim_{\Aff/Y}  g_* f'_* g'^* M \simeq
\lim_{\Aff/Y} f_* g'_* g'^* M \simeq f_* (\lim_{\Aff/Y} g'_* g'^*
M).
$$
To prove the lemma, it thus suffices to show that the natural map
$M\ra \lim g'_* g'^* M$ is an equivalence.

This is a consequence of the stronger claim that the functor $\qc(X)
\ra \lim_{\Aff/Y} \qc({X\times_Y B})$ is an equivalence. Since the
functor $\qc(-)$ takes all colimits of stacks to limits, it
therefore suffices to show that the natural map $X \ra \lim_{\Aff/Y}
(X\times_Y B)$ is an equivalence. This limit can be calculated by
picking an affine cover $U\ra Y$, and realizing $Y$ as the geometric
realization of the usual simplicial object $U_*\to Y$. Finally,
since geometric realization commutes with fiber products we are
done.
\end{proof}

Since $f_+$ was defined to satisfy base change and preserve
colimits, we now have the following.

\begin{proof}[Proof of Proposition \ref{Thomason pushforwards}] The first
assertions were proved above. Since base change is local in the
target, one can prove the final statement for an arbitrary $Y' \ra
Y$ by choosing a cover of $Y'$ by an affine $\Spec A \ra Y'$, thus
reducing to the case which was proved above.
\end{proof}


\subsection{Constructions of perfect stacks}\label{sect properties of air}
In this section, we construct many examples of perfect stacks.

Throughout what follows,
by a derived scheme, we mean a derived stack which admits a Zariski open covering
by affine derived schemes. Recall that
a morphism is Zariski open if it induces a Zariski open morphism
on the underlying truncated underived stacks, as well as isomorphisms of the 
higher homotopy
groups of the structure sheaves over the Zariski open.
Equivalently, one can think of a derived scheme in terms of the underlying truncated
underived scheme equipped with a derived enhancement of the structure sheaf.
Following usual conventions, we say that a derived scheme is quasi-compact if any Zariski open cover admits
a finite refinement. 

We begin with two lemmas needed to show quasi-compact derived schemes with affine diagonal are in fact perfect.

\begin{lemma}\label{colimit preserving} For $X$ a quasi-compact
derived scheme with affine diagonal the global sections functor $\Gamma:\qc(X)\ra \Mod_A$
is colimit preserving.
\end{lemma}

\begin{proof} We briefly sketch the argument
of \cite[Proposition 5.5.5]{dag}. The derived global sections
functor $\Gamma$  preserves finite colimits. Thus it suffices to show $\Gamma$ preserves small
coproducts: we must check that the natural map $\coprod_\alpha
\Gamma(M_\alpha) \ra \Gamma(\coprod_\alpha M_\alpha)$ is an
equivalence, i.e.,~that the induced map on homotopy groups
$\coprod_\alpha \pi_*\Gamma(M_\alpha) \ra \pi_*\Gamma(\coprod_\alpha
M_\alpha)$ is an equivalence. Since $\pi_*$ is colimit preserving,
it suffices to check that the individual terms $\pi_i \Gamma$ each
preserve small coproducts.

This is shown in two steps. First, one checks that for $M\in\qc(X)$
concentrated in a single degree, there exists $m$ such that $\pi_n
\Gamma M$ is zero for all $n< m$. Thus to establish the assertion,
it suffices to work in the subcategory $M\in(\qc(X))^{\leq n+m}_{\geq n}$ for which $\pi_* \Gamma M$ is concentrated in bounded
degrees.

Next, one chooses a finite affine cover $U \ra X$ giving the usual
simplicial object $U_* \ra X$, and thus an identification $\Gamma M
\simeq \lim (M|_{U_q})$.  The resulting Bousfield-Kan, or \u Cech,
spectral sequence has $E_1$-term given by the $\pi_p (M|_{U_q})$, and
converges to $\pi_* \Gamma M$. (The construction of the $E_1$-term
evidently commutes with coproducts in $\qc(X)$, since pullback to an
affine is colimit preserving.) Using the previous step, only finitely
many terms in the spectral sequence may involve differentials
affecting a particular group $\pi_i \Gamma M$.  Therefore we have
expressed $\pi_i \Gamma M$ as a finite limit of terms which preserved
colimits in $M$ and so the result follows.
\end{proof}

The following shows that we can glue derived schemes with finite colimits rather
than geometric realizations. We state it in the simplest form applicable to 
the assertions which follow.

\begin{lemma}\label{lemma fin colim}
Suppose $X$ is a
derived scheme, and $U\coprod V \to  X$ is an open Zariski cover.
Then the following is a colimit diagram
$$
\xymatrix{
U \cap V 
 \ar@<+.5ex>[r] \ar@<-.5ex>[r] &
U\coprod V \ar[r] & X.
}
$$
\end{lemma}

\begin{proof}
It suffices to show that the following diagram of algebra objects is Cartesian
$$
\xymatrix{
\cO_X \ar[d]\ar[r] & \cO_V\ar[d]\\
\cO_U \ar[r]& \cO_{U \cap V}
}
$$

Let $u: U\coprod V\to X$ denote the cover. Since the restriction $u^*(-) \simeq (-) \ot_{\cO_X} (\cO_U \ti \cO_V)$ is conservative and preserves finite limits, it suffices
to show that the restriction of the above diagram is Cartesian.
But this is nothing more than the clearly Cartesian diagram
$$
\xymatrix{
\cO_U \ti \cO_V \ar[d]\ar[r] & \cO_{U\cap V} \ti \cO_{V}\ar[d]\\
\cO_{U} \ti \cO_{U\cap V}  \ar[r]& \cO_{U \cap V} \ti \cO_{U\cap V}
}
$$
\end{proof}

\begin{prop}\label{Neeman redux}
Quasi-compact derived schemes with affine diagonal are perfect.
\end{prop}

\begin{proof}
The result for ordinary (non-derived) schemes $X$ is a theorem of
Neeman \cite{neeman}, extending ideas of Thomason
\cite{thomasontrobaugh}. (In fact, Neeman proves that for
quasi-compact, quasi-separated schemes, $\qc(X)$ is compactly
generated, and dualizable and compact objects coincide. We assume
$X$ has affine diagonal only because the definition of perfect stack
requires it.)

A modified exposition of Neeman's argument appears in the work of
Bondal-Van den Bergh~\cite{BvdB}, who in fact prove that $\qc(X)$ is
generated by a {\em single} perfect object. One can translate the
latter proof, which occupies  \cite[Section 3.3]{BvdB}, directly
into the derived setting, substituting Lemma \ref{colimit
preserving} above for its underived version \cite[Corollary
3.3.4]{BvdB}, and using the natural identification $\qc(\Spec
A)\simeq \Mod_k$ instead of \cite[Corollary 3.3.5]{BvdB}. (In the
derived setting, there is no general notion of the abelian category
of quasi-coherent sheaves, so we do not need to worry about the
potential distinction between its derived category and the
quasi-coherent derived category). In what follows, we sketch the
argument for the reader's convenience, keeping the notation
from~\cite{BvdB}.

The proof that $\qc(X)$ is generated by a single perfect
(dualizable) object is an induction on the number of opens in an
affine cover of $X$. The base case of an affine derived scheme is
Lemma \ref{affine air}. For the inductive step, we write $X=Y\cup U$
with $Y$ open and $U$ affine (putting us in the context of Lemma~\ref{lemma fin colim}), and assume that $\qc(Y)$ has a perfect
generator $E$. By \cite[Proposition 6.1]{BoN}, there is an explicit
compact generator $Q$ for the kernel of the restriction from $U$ to
the intersection $S=Y\cap U$. (One can think of $Q$ as a form of the
structure sheaf of the closed complement $V=U\setminus S$). The key
to the inductive step is Neeman's abstract categorical
form~\cite[Theorem 2.1]{neemanTTY} of Thomason's extension theorem
for compact objects. This allows us to extend $E\oplus E[1]|_S$ to a
compact (hence perfect) object on $U$, and then to glue the latter
to $E\oplus E[1]$ to obtain a perfect object $P$ on all of $X$.
(Note that we extend $E\oplus E[1]$ rather than $E$ itself since
K-theoretic obstructions vanish for the former.) One then checks by
a Mayer-Vietoris argument that the sum of $P$ and the pushforward of
$Q$ (which is itself compact and perfect by support considerations)
to $X$ generates all of $\qc(X)$.

By Lemma \ref{colimit preserving}, we know that $\cO_X$ is compact,
and hence that dualizable complexes are compact. The assertion that
compact objects are dualizable follows from \cite{neemanTTY}: if a
set $\cC^\circ$ of compact objects generates $\cC$, then all compact
objects of $\cC$ are summands of finite colimits of objects of
$\cC^\circ$ and their shifts. Since $\qc(X)$ is generated by a
perfect object, we conclude that all compact objects are summands of
perfect objects, which allows one to check locally that they are
indeed perfect.
\end{proof}

We next make a simple observation about compact objects.

\begin{lemma}\label{compact=dualizable} Suppose $p:X\to Y$ is a
perfect morphism over a perfect base. Then compact and dualizable
objects of $\qc(X)$ coincide.
\end{lemma}

\begin{proof}
First note that pushforward $p_*$ along the perfect morphism $p$ is
colimit preserving, hence (by adjunction) the pullback $p^*M$ of a
compact object $M\in \qc(Y)$ is compact. In particular we find that
the structure sheaf (the monoidal unit) on $X$ is compact, and hence
that all dualizable objects are compact. Note also that the pullback
of a dualizable object is always dualizable.

Now suppose that $U\to Y$ is any affine mapping to $Y$, and consider
the base change $X_U$ of $X$ to $U$. By the definition of a perfect
morphism applied to $p$, this base change is itself a perfect
stack. Let $q:X_U\to X$ denote the base change morphism, which is
affine since $Y$ has affine diagonal. If $M\in \qc(X)$ is any compact
object and $M_U=q^* M$ is its pullback to $U$, then $M_U$ is itself
compact since $q_*$ preserves colimits:
\begin{eqnarray*}
\Hom(M_U, \colim A_i) &\simeq& \Hom(q^* M, \colim A_i)\\
&\simeq & \Hom(M, q_* \colim A_i)\\
&\simeq & \Hom(M, \colim q_* A_i)\\
&\simeq & \colim \Hom(M, q_* A_i)\\
&\simeq & \colim \Hom(M_U, A_i).
\end{eqnarray*}

Since $U$ itself is perfect, it follows that $M_U$ is dualizable and
(by Proposition~\ref{perfect=dualizable}) perfect. We now show that the pullback of $M$ to any affine is perfect. Assume that the map $U\ra Y$ is surjective, so as a consequence $X_U \ra X$ is also surjective. Now let $f: V\ra X$ be any affine mapping to $X$. We may form the Cartesian diagrams:
$$\xymatrix{
U\times_Y V\ar[r]^{f'}\ar[d]^{q'}& X_U\ar[d]^q\ar[r]&U\ar[d]\\
U\ar[r]^f&X\ar[r] &Y\\}$$\noindent
We now verify that $f^* M$ is perfect, given the hypotheses above. The fiber product $U\times_Y V$ is affine, since $Y$ has affine diagonal, and the map $q'$ is surjective since $q$ is. Since $M_U$ is perfect, $f'^*M_U \simeq q'^*f^* M$ is perfect, as well. Thus, in summary, we know that $q'^*f^*M = \cO(V\times_Y U)\ot_{\cO(U)} f^*M$ is perfect, where $q'$ is surjective, and hence $f^* M$ is perfect.
Since the pullback of $M$ to any affine is perfect, therefore $M$ is
itself perfect by Definition \ref{perfect def} and hence (again
by Proposition~\ref{perfect=dualizable}) dualizable.
\end{proof}
%

\begin{prop}\label{ample family} Let $Y$ be a perfect stack and $p:X\to
Y$ a relative quasi-compact derived scheme with affine diagonal and a relatively ample
family of line bundles (for example, $p$ quasi-projective, or in particular affine). Then $X$ is perfect.
\end{prop}

\begin{proof}
By Lemma \ref{compact=dualizable}, we know that compact and
dualizable objects in $\qc(X)$ coincide. Thus we need only to check
that $\qc(X)$ is compactly generated. The pullbacks $p^*M$ of
compact (hence dualizable) objects on $Y$ are dualizable, hence
compact, as are the line bundles $\cL$ in the given relatively ample
family. We claim the $\infty$-category of compact objects $p^*M\ot\cL$
generates $\qc(X)$. The argument is as above in the case of an
external product: let $N$ be right orthogonal to the compact
objects, so that $\Hom(p^*M\ot \cL,N)=0$. We first find by
adjunction and the fact that $Y$ is compactly generated that
$p_*\intHom(\cL,N)=0$. Since the objects $\cL$ form a relatively ample
family of line bundles this forces $N=0$.
\end{proof}

\begin{cor}\label{quasiproj}
In
characteristic zero,
the quotient $X/G$ of a quasi-projective derived scheme $X$ by a linear
action of an affine algebraic group $G$ is perfect.
\end{cor}

\begin{proof}
First note that in
characteristic zero,
 $BGL_n$ is clearly perfect: the compact and dualizable objects are both
finite dimensional representations which generate. 

If $G$ is an
affine algebraic group, we can embed $G\hookrightarrow GL_n$ as a
subgroup of $GL_n$ for some $n$.  Thus we obtain a morphism $BG \to
BGL_n$ with fiber $GL_n/G$. 
By a theorem of Chevalley \cite{chevalley}, $GL_n/G$ is a quasi-projective variety,
and so by
Proposition \ref{ample family}, $BG$ itself is perfect. 

Finally,
for a quasi-projective derived scheme $X$ with a linear action of $G$, the
morphism $X/G\to BG$ is quasi-projective, so applying Proposition \ref{ample family} again,
we conclude that $X/G$ is perfect.
\end{proof}

\begin{cor}\label{airated} Let $p:X\to Y$ be a morphism between perfect stacks. Then $p$ is perfect.
\end{cor}

\begin{proof}
Let $f:A\to Y$ be a morphism from an affine derived scheme to the
base, and $X_A\to A$ the base change of $X$. Since $Y$ has affine
diagonal, $f$ is affine as is the base change morphism $f_X:X_A\to
X$. In particular Proposition \ref{ample family} (again in the basic
case of an affine morphism) applies to $f_X$, so that the total
space $X_A$ is perfect.
\end{proof}

\begin{prop}\label{Thomason product} The product $X=X_1\times X_2$ of perfect stacks
is perfect. More generally, for maps $p_i:X_i\to Y$, if $Y$ has affine diagonal,
 then $X_1\times_Y X_2$ is perfect.
\end{prop}

\begin{proof}
The second assertion follows from the first and Proposition
\ref{ample family}, since $X_1\times_Y X_2\to X_1\times X_2$ is an
affine morphism for $Y$ with affine diagonal.

To prove the first assertion, note that $X=X_1\times X_2$ has affine diagonal. 
Applying Lemma \ref{compact=dualizable} to the
projection to a factor, we see that compact and dualizable objects
of $\qc(X)$ coincide. Thus to confirm that $X$ is perfect, it suffices
to show that $\qc(X)$ is compactly generated.

Let us first check that the external product
of compact objects is again compact. By assumption, compact objects
$M_i$ on the factors $X_i$ are dualizable, so $M_1\boxtimes M_2$ is
dualizable (it is the tensor product of pullbacks, and both
operations preserve dualizabie objects), and hence compact.

Now let us check that external products of compact objects generate
$\qc(X)$. The argument is a modification of the argument of
Bondal-Van den Bergh \cite{BvdB} in the case of a single compact
generator. Namely, let $N$ be right orthogonal to $\qc(X)^c$, so
that in particular $\Hom(M_1\boxtimes M_2,N)\simeq 0$ for all
$M_i\in \qc(X_i)^c$. By adjunction, we have
\begin{eqnarray*} 0 &\simeq & \Hom(M_1\boxtimes M_2, N)\\
&\simeq &\Hom(\pi_1^*M_1,\intHom(\pi_2^*M_2,N))\\
&\simeq &\Hom(M_1,\pi_{1,*}\intHom(\pi_2^*M_2,N))
\end{eqnarray*}
for all $M_1,M_2$, so that $\pi_{1*}\intHom(\pi_2^*M_2,N)\simeq 0$
since such $M_1$ generate $\qc(X_1)$. For any affines $U\to X_1$ and
$V\to X_2$, we therefore have
\begin{eqnarray*} 0 &\simeq & \Gamma(U,\pi_{1,*}\intHom(\pi_2^*M_2,N))\\
&\simeq & \Hom_{U\ti X_2}(\pi_2^* M_2, N)\\
&\simeq & \Hom_{X_2}(M_2, (\pi_2|_{U\ti X_2})_*N)
\end{eqnarray*}
for all $M_2$. Since the latter objects generate $\qc(X_2)$ it
follows (upon restricting to $V$) that $\Gamma(U\times V, N)\simeq
0$, whence (by affineness of $U\ti V$) that $N|_{U\times V}\simeq
0$, and finally (since affines of the form $U\times V$ cover $X_1\ti
X_2$) that $N\simeq 0$.
\end{proof}

\begin{cor}\label{cor finite mapping stacks have air}
Let $X$ be a perfect stack, and let $\Sigma$ be a finite simplicial
set. Then the mapping stack $X^\Sigma = \Map(\Sigma, X)$ is perfect.
\end{cor}

\begin{proof}
Let $\Sigma_0$ be the $0$-simplices of $\Sigma$. Since $X$ has
affine diagonal, the natural projection $X^\Sigma\to X^{\Sigma_0}$
is affine. By Proposition~\ref{Thomason product}, the product
$X^{\Sigma_0}$ is perfect, and so by Proposition~\ref{ample family} (in
the basic case of an affine morphism), the assertion follows.
\end{proof}

Finally, we consider arbitrary quotients by finite group schemes in
good characteristics.

\begin{prop} Let $X$ be a perfect stack with an action of an affine 
group scheme $G$ for which

\begin{enumerate}
\item The global functions $\Gamma(G, \cO_G)$ is a perfect
  complex. 
\item The unit $\cO_{BG}$ on $BG$ is compact
  (equivalently, the trivial $G$-module is perfect).
\end{enumerate}

Then $X/G$ is perfect.
\end{prop}

\begin{proof}
We leave to the reader the exercise of checking that $X/G$ has affine diagonal since $X$ has affine diagonal
and $G$ is affine.

We will first prove that condition (1) implies $\qc(X/G)$ is generated
by compact dualizable objects.  
Since $f:X\to X/G$ is affine, we have the identification $\qc(X) \simeq \Mod_{ f_*\cO_X}(\qc(X/G))$.

We claim that the algebra object $ f_*\cO_X$ is perfect (or equivalently, dualizable).
To see this, consider the pullback square of
derived stacks
$$ \xymatrix{ G\ar[r]^p\ar[d]_p&\Spec k\ar[d]^g\\ \Spec k \ar[r]^g&
  BG\\}$$
Via base change, we obtain an equivalence $g^*g_*\cO_k\simeq
p_*p^*\cO_{k}$, or in other words, an equivalence of $k$-algebras
$g^*g_*\cO_k \simeq \Gamma(G, \cO_G)$.  By assumption, $\Gamma(G,
\cO_G)$ is a perfect complex, and $g^*$ is conservative and preserves
perfect complexes, so we conclude that the pushforward $g_*\cO_k$ is
perfect.
Now consider the pullback square of derived stacks
$$
\xymatrix{
X\ar[r]^f\ar[d]_{q'}&X/G\ar[d]^q\\
\Spec k \ar[r]^g& BG\\}$$
Since $g_* \cO_k$ is perfect, $q^*g_*\cO_k$ is perfect. By base
change, we have the equivalence $q^*g_*\cO_k \simeq
f_*q'^*\cO_k$, and thus we conclude that $f_*\cO_X$ is perfect.

Next observe that the right adjoint $f^+$ to the pushforward $f_*$ can be calculated explicitly by
$$
f^+(M) \simeq \intHom_{\cO_{X/G}}(f_*\cO_X, M) \simeq M \ot_{\cO_{X/G}} (f_*{\cO_X})^\vee.
$$
It follows immediately that $f^+$ preserves colimits.
It also follows that $f^+$ is conservative since a diagram chase with the above identities leads to the identity
$$
f^*( M \ot_{\cO_{X/G}} (f_*{\cO_X})^\vee) \simeq f^*(M) \ot_{\cO_{X}} q'^*g^*((g_*{\cO_k})^\vee)
\simeq 
f^*(M) \ot_{\cO_X} (\cO_X \ot_{\cO_k} \Gamma(G, \cO_G)^\vee).
$$ 
The unit  $e:\Spec k \to G$ gives a factorization of the identity map
$$
\xymatrix{
\cO_k \ar[r]^-{p^*} &  \Gamma(G, \cO_G) \ar[r]^-{e^*} & \cO_k,
}
$$
and taking duals, a factorization of the identity map of $\cO_k$ through the dual $\Gamma(G, \cO_G)^\vee$.
Hence if $f^+(M)$ were trivial, then $f^*(M)$ would also be trivial, but $f^*$ is conservative. 
Thus we conclude
$f_*$ takes a generating set of compact objects to a generating set of compact objects.

%


We now appeal to condition $(2)$ that the unit in $\qc(BG)$ is compact,
hence so are all dualizables in $\qc(BG)$. In the case  when $X$ is a point, 
the above arguments show that $\qc(BG)$ is compactly generated.
Furthermore, it shows that all compacts are in fact dualizable (since $f_*\cO_k$
is a compact dualizable generator), and hence $BG$
itself is perfect. The morphism $X/G\to BG$ is then a perfect morphism
with perfect base. Thus by Lemma \ref{compact=dualizable} compact and
dualizable objects in $\qc(X/G)$ coincide. This implies (in
combination with the compact generation of $\qc(X/G)$ above) that
$X/G$ is perfect as asserted.
\end{proof}




\section{Tensor products and integral transforms}\label{integral}
In this section, we study the relation between the geometry and
algebra of perfect stacks. We begin 
with some basic properties of tensor
products of $\infty$-categories. Then we prove (Theorem \ref{Thomason tensor})
that the $\infty$-category of sheaves on a product of perfect stacks,
 and more generally a derived fiber product, 
is given by the tensor product of the
$\infty$-categories of sheaves on the factors. This implies that the
$\infty$-category of sheaves on a perfect stack is self-dual, which in turn  allows us to describe
$\infty$-categories of functors by $\infty$-categories of integral kernels (Corollary
\ref{transforms with air}). Finally,
we extend this last result to
a relative setting where the base is an arbitrary derived stack with affine diagonal.

\subsection{Tensor products of $\infty$-categories}\label{tensors of cats}
In this section, we consider some properties of tensor products of $\infty$-categories
that will be used in what follows. First we prove 
(Proposition \ref{dualizable prop}) that $\infty$-categories of modules
over associative algebra objects are dualizable, with duals given by modules for the
opposite algebra, and that tensor product of algebras induces tensor products on module categories.
We then discuss the tensor product of small stable categories, and its compatibility
with passing to the corresponding presentable stable $\infty$-categories of $\Ind$-objects. 

\subsubsection{Algebras and modules}\label{basic algebra}
Recall the tensor product of presentable stable $\infty$-categories
developed in~\cite{dag2}, \cite{dag3}.  Namely, the $\infty$-category
$\mathcal Pr^{\rm L}$ of presentable $\infty$-categories (with
morphisms given by left adjoints) carries a natural symmetric monoidal
tensor product that preserves stable objects.

Let $\cC$ be a monoidal $\infty$-category.  For two presentable
$\oo$-categories $\cM$ and $\cM'$ left tensored over $\cC$, we denote
by $\Fun^{\rm L}_\cC(\cM, \cM')$ the $\oo$-category of left adjoints
from $\cM$ to $\cM'$ that preserve the tensor over $\cC$.

Similarly, the opposite category of $\mathcal Pr^{\rm L}$ is the
$\infty$-category $\mathcal Pr^{\rm R}$ of presentable
$\infty$-categories with morphisms given by right adjoints.  For two
presentable $\oo$-categories $\cM$ and $\cM'$ left cotensored over
$\cC$, we denote by $\Fun^{\rm R}_\cC(\cM, \cM')$ the $\oo$-category
of right adjoints from $\cM$ to $\cM'$ that preserve the cotensor over
$\cC$.

\begin{prop}\label{dualizable prop}
Let $\cC$ be a stable presentable symmetric monoidal $\infty$-category,
and $A\in \cC$ an associative algebra object. 

\begin{enumerate}
\item 
For any $\cC$-module $\cM$, there is a canonical
equivalence of $\infty$-categories
$$
\Mod_A(\cC) \ot_\cC \cM \simeq \Mod_A(\cM).
$$

\item
For $A'\in \cC$ a second associative algebra, there is a 
 canonical
equivalence of $\infty$-categories
$$
\Mod_{A\otimes A'}(\cC) \simeq \Mod_{A}(\cC)\otimes_\cC
\Mod_{A'}(\cC).
$$

\item The $\infty$-category of modules $\Mod_A(\cC)$ is dualizable as a $\cC$-module
with dual given by the $\infty$-category of modules  $\Mod_{A^{\rm op}}(\cC)$ over the opposite algebra.
\end{enumerate}

\end{prop}

\def\m{\mathrm{Mod}}

The proof will depend on the following two lemmas:

\begin{lemma}\label{useful lemma}
Let $\cM$ and $\cM'$ be stable presentable $\oo$-categories that are 
left tensored and cotensored over $\cC$.
Let $G: \cM' \ra \cM$ be a right adjoint that is tensored and cotensored over $\cC$. Assume further that $G$ is colimit preserving. 
Then $G$ is conservative if the induced functor $\Fun_\cC^{\rm R}(\cD, \cM') \ra \Fun_\cC^{\rm R}(\cD, \cM)$ is conservative
for any $\cD$.
\end{lemma}
\begin{proof}
Suppose $G$ is not conservative. Then to prove the lemma, it suffices to exhibit a
presentable $\oo$-category $\cD$ also
cotensored over $\cC$ and a nontrivial right adjoint $j: \cD \ra \cM'$
cotensored over $\cC$ such that $j\circ G$ is trivial.

Define $\cD$ to be the full $\oo$-subcategory of $\cM'$ of $G$-acyclic objects,
that is, 
objects $m \in \cM$ such that $G(m)$ is trivial. Our first
task is to show that $\cD$ is indeed presentable.

Observe that $\cD$ is equivalent to the fiber product $\cD\simeq 0 \times_\cM \cM'$, where the limit is computed in the $\oo$-category $\rm Cat_\oo$
of $\oo$-categories. Recall by \cite[Proposition 5.5.3.13]{topos}, the natural functor
$\mathcal Pr^{\rm L} \ra {\rm Cat}_\infty$ preserves limits. Furthermore, 
the forgetful functor $\m_\cC(\mathcal Pr^{\rm L}) \to \mathcal Pr^{\rm L}$ also preserves limits since it has a left adjoint (given by induction). 

Since the functor $G$ preserves colimits and is $\cC$-linear, we may regard it as a morphism in $\m_\cC(\mathcal Pr^{\rm L})$. Thus $\cD$ can be computed as a limit in $\m_\cC(\mathcal Pr^{\rm L})$, 
and so can be regarded as an object of $\mathcal Pr^{\rm L}$.
In other words, 
 $\cD$ is presentable and furthermore tensored over $\cC$.
Finally, since $\cD$ is tensored over $\cC$, it is automatically cotensored as well.

Now it remains to show that the inclusion $j:\cD \ra \cM'$ is indeed a right adjoint
and cotensored over $\cC$. 
Since $j$ preserves all limits and colimits (and in particular $\kappa$-filtered colimits), 
the adjoint functor theorem applies. Finally, since $G$
is cotensored over $\cC$,  $j$ is as well.
\end{proof}

\begin{lemma}\label{prop cons} 
Let $\cM$ be a stable presentable $\oo$-category which is left tensored and cotensored over $\cC$, and let $A$ be an associative algebra in $\cC$. Then the forgetful functor
$G:  \m_A(\cC) \ot_\cC \cM \to \cM$ 
is conservative.
\end{lemma}

\begin{proof} Observe that  for any $\cD$ tensored over $\cC$, the pullback
$$
\Fun^{\rm L}_\cC (\m_A(\cC) \ot_\cC \cM, \cD) \ra \Fun^{\rm L}_\cC(\cM, \cD)
$$ induced by the induction $F: \cM \to \m_A \ot_\cC \cM$ is
conservative. In other words, if a functor out of $\m_A \times \cM$
(which preserves colimits in each variable) is trivial when restricted
to $\cM$, then it is necessarily trivial.

Consequently, switching to opposite categories, we have that the corresponding functor
$$\Fun^{\rm R}_\cC (\cD, \m_A \ot \cM) \ra \Fun^{\rm R}_\cC (\cD, \cM)
$$ 
induced by the forgetful functor $G:  \m_A \ot_\cC \cM \to \cM$
 is conservative.

Now we can apply Lemma \ref{useful lemma} with $\cM' =   \m_A \ot \cM$
to obtain that $G$ is conservative.
\end{proof}


\begin{proof}[Proof of Proposition \ref{dualizable prop}]

We first prove that $\m_A(\cC)\ot_\cC \cM$ is equivalent to
$\m_A(\cM)$ by the natural evaluation functor.  Consider the
adjunction
$$
\xymatrix{\cC\ar@/^1pc/[r]^-{F}&\m_{A}(\cC) \ar[l]^-{G}
}
$$
where $F(-) =A\ot -$ is the induction, and $G$ is the forgetful functor.

The above adjunction induces an adjunction
$$
\xymatrix{\cM\ar@/^1pc/[r]^-{F\ot \Id}&\m_A(\cC) \ot_\cC \cM \ar[l]^-{G\ot\Id}
\ar[r] &\m_T(\cM)
}
$$ and thus a functor to modules over the monad $T=(G\ot{\rm
  id})\circ(F\ot{\rm id})$ acting on $\cM$.  The functor underlying
$T$ is given by tensoring with $A$, so we also have an equivalence
$\m_T(\cM)\simeq \m_A(\cM)$.

By its universal characterization, the functor
$G\ot {\rm id}$ is colimit preserving.
Note as well that $G$ and hence $G\ot\Id$ is also $\cC$-linear
(or in other words, the adjunction satisfies an analogue of the projection formula).
Thus it follows from Lemma~\ref{prop cons} that $G\ot{\rm id}$ is also conservative.
Thus $G\ot{\rm id}$ satisfies the monadic Barr-Beck conditions, and we obtain the
desired equivalence $\m_A(\cC) \ot_\cC \cM\simeq \m_A(\cM)$.

Next, we can apply this to the instance where $\cM$ is the
$\infty$-category of left modules over another associative algebra
$A'$ to conclude that there is a natural equivalence $\m_A(\cC)\ot_\cC
\m_{A'}(\cC) \simeq \m_A(\m_{A'}(\cC))$. We now have a chain of
adjunctions
$$
\xymatrix{\cC\ar@/^1pc/[r]^-{F'}&\m_{A'}(\cC)\ar@/^1pc/[r]^-{F''} \ar[l]^-{G'}&\m_A(\m_{A'}(\cC))\ar[l]^-{G''}\\}
$$ in which the composite $G'\circ G''$ is colimit preserving and
conservative, and hence satisfies the monadic Barr-Beck conditions.

Furthermore, the above adjunction naturally extends to a diagram in which the cycle of left adjoints (denoted
by bowed arrows), and hence
also the cycle of right adjoints (denoted by straight arrows),  commute
$$
\xymatrix{\ar@/_1pc/[dr]_-{F} \cC\ar@/^1pc/[rr]^-{F''F'} &&  \m_A(\m_{A'}(\cC))\ar[ll]^-{G'G''}
\ar@/^1pc/[dl]^-{f} 
\\
& \ar[ul]_-G \m_{A\ot A'}(\cC) \ar[ur]^-{g} &
}
$$
Here $F(-) =A_1\ot A_2 \ot -$ is the induction, $G$ is the forgetful functor,
$f$ is the natural functor factoring through $\m_A(\cC)\ot_\cC\m_{A'}(\cC)$,
and $g$ is its right adjoint. From this diagram, we obtain a morphism of monads
$$
\xymatrix{
G'G''F''F' \ar[r] &  G'G''gfF''F' \simeq GF.
}
$$

Now the underlying functors of the monads $GF(-)$ and $G'G''F''F'(-)$ are both equivalent to the tensor $A\ot A'\ot(-)$, so 
the above morphism of monads is an equivalence. Thus we obtain the promised equivalence
$\m_A(\cC)\ot_\cC\m_{A'}(\cC) \simeq \m_A(\m_{A'}(\cC))\simeq \m_{A\ot
  A'}(\cC)$. 

Finally, we show that the $\infty$-category of left $A$-modules
$\m_A(\cC)$ is a dualizable $\cC$-module by directly exhibiting the
$\infty$-category of right $A$-modules $\m_{A^{\rm op}}(\cC)$ as its
dual. The trace map is given by the two-sided bar construction
$$
\tau:  \m_A(\cC) \ot_\cC\m_{A^{\rm op}}(\cC) \to \cC
\qquad
M, N \mapsto M\ot_A N
$$ 
The unit map  is given by the induction
$$
u:\cC \to \m_{A^{\rm op}}(\cC) \ot_\cC \m_A(\cC) 
\simeq \m_{A^{\rm op}\ot A}(\cC)
\qquad
c \mapsto A\ot c
$$
where we regard $A\ot c$ as an $A$-bimodule.

One can verify directly that the composition
$$
\xymatrix{\m_A(\cC)\ar[r]^-{{\rm id}\ot u}&\m_A(\cC)\ot_\cC \m_{A^{\rm op}}(\cC)\ot_\cC\m_A(\cC)
\ar[r]^-{\tau\ot{\rm id}} &\m_A(\cC)}
$$ 
is equivalent to the identity. First, $({\rm id}\ot u)(M)$ is equivalent to $A \ot M$ 
regarded as an $A\ot A^{\rm op} \ot A$-module, and
second, $(\tau\ot{\rm id})(A\ot M)$ is equivalent to $A\ot_A M \simeq M$.
\end{proof}


\subsubsection{Small stable categories}
We have been working with the symmetric monoidal structure on the $\oo$-category 
$\mathcal Pr^{\rm L}$ of presentable
$\infty$-categories as developed in~\cite{dag2},
\cite{dag3}.
We will also need the tensor product of small stable idempotent complete
$\infty$-categories,
in particular, the $\infty$-categories of compact objects in presentable stable $\infty$-categories. 

Let $\Kar$ be the full $\infty$-subcategory of the $\infty$-category
of stable categories (with morphisms exact functors) consisting of
those $\infty$-categories that are idempotent complete. Recall that an
$\oo$-category $\cC$ is idempotent complete if the essential image of
the Yoneda embedding $\cC \ra \cP(\cC)$ is closed under retracts.

\begin{prop}\label{Karoubian}
The $\infty$-category $\Kar$ carries a symmetric monoidal structure characterized by the property
that for $\cC_1,\cC_2,\cD\in \Kar$, the $\infty$-category of exact functors
$\Fun_{\Kar}(\cC_1\ot \cC_2, \cD)$ is equivalent to the full $\infty$-subcategory 
of all functors $\cC_1\times \cC_2 \to \cD$ that preserve finite colimits in 
$\cC_1$ and $\cC_2$ separately. Furthermore, passing to the corresponding 
stable presentable $\oo$-categories of $\Ind$-objects 
is naturally a symmetric monoidal functor.
\end{prop}

\begin{proof}
For $\cC_1,\cC_2\in \Kar$, we define their tensor product by
$$
\cC_1\ot \cC_2 = (\Ind(\cC_1)\ot \Ind(\cC_2))^c
$$ 
where the tensor product of the right hand side  is calculated in the $\oo$-category
$\mathcal Pr^{\rm L}$ of presentable $\oo$-categories (with morphisms left adjoints), and
the superscript ${}^c$ denotes the full $\infty$-subcategory of compact objects of a presentable $\oo$-category.
Since $\Ind\cC_1\ot \Ind\cC_2$ is idempotent complete
and  retracts of compact objects are compact,
$(\Ind\cC_1\ot \Ind\cC_2)^c$ is idempotent complete as well.
Thus the tensor product $\cC_1\ot \cC_2$ is indeed an object of $\Kar$.

For $\cC_1,\cC_2,\cD\in \Kar$, let $\Fun'(\cC_1 \times \cC_2, \cD)$ be the full $\oo$-subcategory of 
functors $\cC_1 \times \cC_2\to \cD$ that preserve finite colimits in 
$\cC_1$ and $\cC_2$ separately.
We claim that For $\cC_1,\cC_2\in \Kar$, the tensor product
$\cC_1\ot \cC_2$ corepresents the functor $\Fun'(\cC_1 \times \cC_2, -)$
in the sense that for any $\cD\in\Kar$, there is a canonical equivalence
$$
\Fun'(\cC_1 \times \cC_2, \cD) \simeq \Fun_{\Kar}(\cC_1 \ot \cC_2, \cD).
$$
As a consequence, the associativity and symmetry of the tensor product $\cC_1\ot\cC_2$ 
will immediately follow from the analogous properties of $\Fun'$.

For $\cC_1,\cC_2,\cD\in \mathcal Pr^{\rm L}$, let $\Fun^{L \ti L}(\cC_1\times \cC_2,\cD)$ be 
the full $\oo$-subcategory of 
functors $\cC_1 \times \cC_2\to \cD$ that preserve colimits in 
$\cC_1$ and $\cC_2$ separately.
To prove the claim, observe that the inclusion $\cD \ra \Ind\cD$ induces a fully faithful functor 
$$\Fun'(\cC_1\times \cC_2, \cD) \ra \Fun'(\cC_1\times \cC_2, \Ind\cD)
\simeq \Fun^{L \ti L}(\Ind\cC_1\times \Ind\cC_2, \Ind\cD)
$$ 
Its essential image consists of functors that preserve compact objects.
 By definition of the monoidal structure on the $\oo$-category $\mathcal Pr^{\rm L}$
 of presentable $\oo$-categories, we have a further equivalence 
 $$
 \Fun^{L\ti L}(\Ind\cC_1\times \Ind\cC_2, \Ind\cD) \simeq \Fun^{\rm L}(\Ind\cC_1\ot \Ind\cC_2, \Ind\cD).
 $$ 

Since the compact objects of $\Ind\cC_1 \ot \Ind\cC_2$ are generated by finite colimits of external products of compacts objects, we obtain an equivalence between 
$\Fun'(\cC_1\times \cC_2, \cD)$ and the full $\oo$-subcategory of $\Fun^{\rm L}(\Ind\cC_1\ot\Ind \cC_2, \Ind\cD)$ consisting of functors that preserve compact objects.
In other words, 
we have the asserted equivalence that characterizes the tensor product
$$
\Fun'(\cC_1\times \cC_2, \cD) \simeq \Fun_{\Kar}((\Ind\cC_1\ot\Ind \cC_2)^c, \cD).
$$

Finally, the assertion that the functor $\Ind:\Kar\to \mathcal Pr^{\rm L}$ 
is symmetric monoidal is immediate from the constructions
and the natural equivalence $\Ind(\cC^c) \simeq \cC$, for $\cC\in \mathcal Pr^{\rm L}$.
\end{proof}

\begin{remark}
Given small stable idempotent complete
$\infty$-categories $\cC_1, \cC_2$,
by construction
their tensor product $\cC_1\ot \cC_2$ is again 
a small stable idempotent complete
$\infty$-category. Though it is possible to consider other versions of a tensor
product on small stable $\infty$-categories that need not preserve idempotent complete
$\infty$-categories,
our approach builds it in from the beginning.
\end{remark}


\subsection{Sheaves on fiber products}
In this section, we study the $\infty$-category of sheaves on the derived
fiber product of perfect stacks. The main technical result is
that it is equivalent to the tensor product of the $\infty$-categories of sheaves
on the factors
(Theorem \ref{Thomason tensor}). The proof involves first showing that an analogous
assertion holds for $\infty$-categories of perfect complexes.
 The remainder of the section is devoted to collecting corollaries of the main techinical result.

Recall that for a perfect stack $X$, 
compact (equivalently, perfect or dualizable) objects in the $\infty$-category $\qc(X)$
form a small stable idempotent complete $\infty$-category $\qc(X)^c$, and there is canonical
equivalence
$\qc(X)\simeq \Ind \qc(X)^c$.
Recall as well the tensor product of small stable idempotent complete $\oo$-categories, 
and that the functor $\Ind$ is symmetric monoidal.

\begin{prop}\label{prop external product}
Let $X_1, X_2$ be perfect stacks. Then external tensor product
defines an equivalence
$$
\boxtimes:\qc(X_1)^c \ot \qc(X_2)^c\risom \qc(X_1\ti X_2)^c
$$
In other words, the $\oo$-category of perfect complexes on the product is the 
(small stable idempotent complete) tensor product of the
$\oo$-categories of perfect complexes on the factors.
\end{prop}

\begin{proof}
Set $X=X_1\times X_2$. 
By Proposition
\ref{Thomason product}, we know that the external product takes compact objects
to compact objects, and $\qc(X)^c$ is generated by external products.

Thus it suffices to verify that for $M_i, N_i\in \qc(X_i)^c$ we have an
equivalence
$$\Hom_{X}(M_1\boxtimes M_2, N_1\boxtimes N_2) \simeq \Hom_{X_1}(M_1,N_1)
\ot \Hom_{X_2} (M_2,N_2).$$ 
Using the fact that each $M_i$ is dualizable
and $p_2$ satisfies the projection formula (since it is perfect), we calculate
\begin{eqnarray*}
\Hom_{X}(p_1^* M_1\ot p_2^* M_2, p_1^* N_1\ot p_2^* N_2) &\simeq &
\Gamma(X, p_1^* M_1^\vee \ot p_1^* N_1 \ot p_2^* M_2^\vee \ot p_2^* N_2)\\
&\simeq & \Gamma(X_2, (p_2)_*( p_1^*\intHom_{X_1}(M_1,N_2)\ot
p_2^*\intHom_{X_2}(M_2,N_2)))\\
&\simeq & \Gamma(X_2, \Hom_{X_1}(M_1,N_1)\ot \intHom_{X_2}(M_2,N_2))\\
&\simeq & \Hom_{X_1}(M_1,N_1)\ot \Hom_{X_2}(M_2,N_2)
\end{eqnarray*}
\end{proof}

We now prove our main theorem which identifies $\oo$-categories of sheaves
on fiber products algebraically.  The proof relies on the following
consequence \cite[Corollary 5.5.3.4]{topos} of the $\oo$-categorical
adjoint functor theorem: there is a canonical equivalence
$$({\mc Pr}^{\rm L})^{op} \simeq {\mc Pr}^{\rm R}$$ between the opposite of the
$\oo$-category $\mc Pr^{\rm L}$ of presentable $\oo$-categories with morphisms
left adjoints and the $\oo$-category $\mc Pr^{\rm R}$ of presentable
$\oo$-categories with morphisms right adjoints. In other words, we can
reverse diagrams of presentable $\oo$-categories, in which the
functors are all left adjoints, by passing to the corresponding right
adjoints.

We will also use the fact that the calculation of small limits of
presentable $\oo$-categories is independent of context. Namely,
by~\cite[Proposition 5.5.3.13, Theorem 5.5.3.18]{topos}, the forgetful
functors from ${\mc Pr}^{\rm L}$, ${\mc Pr}^{\rm R}$ to all
$\oo$-categories preserves small limits. In particular, given a small
diagram of both left and right adjoints, the universal maps from the
limit to the terms of the diagram are also both left and right
adjoints.

\begin{theorem}\label{Thomason tensor}
Let $X_1$, $X_2$, $Y$ be perfect stacks with maps $p_1:X_1\to Y$, $p_2:X_2\to Y$. Then
there is
a canonical equivalence
$$
\qc(X_1)\ot_{\qc(Y)} \qc(X_2) \risom \qc(X_1\times_Y X_2).
$$
\end{theorem}

\begin{proof}

To begin, consider the case $Y=\Spec k$, so $X_1\times_Y X_2=
X_1\times X_2$. 
By Proposition~\ref{prop external product} and 
 the fact that $\Ind:\Kar\to \mathcal Pr^{\rm L}$ is symmetric monoidal, 
the external product functor provides an equivalence
$\boxtimes:\qc(X_1)\ot \qc(X_2)\risom \qc(X_1\times X_2)$.

To begin the case of a general perfect stack $Y$, 
consider the augmented cosimplicial diagram
$$
\xymatrix{
X_1\ti_Y X_2 \ar[r]^-\pi & X_1\times X_2
 \ar@<+.5ex>[r] \ar@<-.5ex>[r] &
 X_1\times Y \times X_2 
 \ar@<+1ex>[r] \ar[r] \ar@<-1ex>[r] &
X_1\times Y\times Y \times X_2 
\cdots }$$
 with the obvious maps constructed from the given maps $p_1, p_2$.

Applying the (contravariant) functor $\qc$, we obtain an augmented simplicial $\oo$-category 
$$
\xymatrix{
\qc(X_1\ti_Y X_2) & \ar[l]_-{\pi^*}  \qc(X_1\times X_2)
&  \ar@<+.5ex>[l] \ar@<-.5ex>[l] 
 \qc(X_1\times Y \times X_2) &
 \ar@<+1ex>[l] \ar[l] \ar@<-1ex>[l] 
 \cdots }$$ with structure maps given by pullbacks.  By the absolute
case of the theorem when $Y=\Spec k$, if we forget the augmentation,
we obtain the simplicial $\infty$-category with simplices
$$\qc(X_1)\ot \qc(Y)\ot\cdots \ot\qc(Y)\ot \qc(X_2)$$ 
and structure maps given by tensor contractions. 
This is precisely 
 the two-sided bar construction
\cite[4.5]{dag2} whose geometric realization, by definition~\cite[5]{dag3},
calculates the tensor product
of $\qc(Y)$-modules
$\qc(X_1)\ot_{\qc(Y)}\qc(X_2)$.
Furthermore, the augmentation provides the natural map
$$
\xymatrix{
\qc(X_1\ti_Y X_2) & \ar[l]_-{\tilde\pi^*} \qc(X_1)\ot_{\qc(Y)}\qc(X_2)
}
$$
which we will prove is an equivalence.

The above geometric realization is a colimit in $\mc Pr^{\rm L}$, and hence (as
observed prior to the statement of the theorem) may be evaluated as a limit in the opposite category
$\mc Pr^{\rm R}$.  Thus we find that $\qc(X_1)\ot_{\qc(Y)} \qc(X_2)$ is also the
totalization of the cosimplicial $\oo$-category
$$
\xymatrix{
 \qc(X_1\times X_2)
  \ar@<+.5ex>[r] \ar@<-.5ex>[r] &
 \qc(X_1\times Y \times X_2) 
 \ar@<+1ex>[r] \ar[r] \ar@<-1ex>[r] &
 \cdots }$$ with structure maps given by pushforwards. (We note for
future reference that these structure maps are pushforwards along
affine morphisms, hence are also colimit preserving, i.e., left
adjoints.)

In
particular, pushforward along the augmentation provides a natural
functor
$$
\xymatrix{
\qc(X_1\times_Y X_2) \ar[r]^-{\tilde\pi_*} & \qc(X_1)\ot_{\qc(Y)} \qc(X_2)
}$$
which is an equivalence if and only if $\tilde\pi^*$ is an equivalence.

To summarize some of the above structure, we have a diagram of commuting left (lower arrows) and right (upper arrows) adjoints
$$
\xymatrix{
\qc(X_1\ti_Y X_2)
\ar@/^2pc/[rr]^-{\pi_*}
\ar@<+.5ex>[r]^-{\tilde\pi_*}&
\qc(X_1)\ot_{\qc(Y)} \qc(X_2)
\ar@<+.5ex>[r]^-{\tau_*} \ar@<+.5ex>[l]^-{\tilde\pi^*}&\qc(X_1 \ti X_2) \ar@<+.5ex>[l]^-{\tau^*}
\ar@/^2pc/[ll]^-{\pi^*}
\\}
$$ 
where $\tau_*$ is the universal map from the totalization to the zero cosimplices,
and likewise,  $\tau^*$ is the universal map from the zero simplices to the geometric realization.
Thus we obtain a map of monads 
$$
\xymatrix{
T_{alg} = \tau_*\tau^* \ar[r] & T_{geom} = \pi_* \pi^* 
}
$$
acting on $\qc(X_1 \ti X_2)$.

The geometric pushforward $\pi_*$ is conservative and preserves
colimits since $\pi$ is affine.  Hence by the Barr-Beck theorem, we
have a canonical equivalence
$$
\qc(X_1 \ti_Y X_2) \simeq \Mod_{T_{geom}}(\qc(X_1 \ti X_2)).
$$

We also claim that the universal map $\tau_*$ is conservative and
preserves colimits. For the first assertion, recall that $\tau_*$ is
nothing more than the forgetful map from the totalization to the zero
cosimplices. Since the $\oo$-categories involved are all stable,
evaluating conservatism of a functor is equivalent to determining if
nonzero objects are sent to zero. But an object sent to zero in the
zeroth cosimplices is sent to zero in all cosimplices, and hence is
equivalent to the zero object in the limit $\oo$-category.

To see $\tau_*$ preserves colimits, recall that the structure maps of
our cosimplicial diagram are both right and left adjoints. It then
follows (as observed prior to the statement of the theorem) that the
totalization may be evaluated equivalently back in the category $\mc
Pr^{\rm L}$. In particular, the universal functor $\tau_*$ is a
morphism in $\mc Pr^{\rm L}$, and hence a left adjoint and so
preserves colimits.  

We may now apply the Barr-Beck theorem, giving a canonical equivalence
$$
\qc(X_1)\ot_{\qc(Y)} \qc(X_2)
 \simeq \Mod_{T_{alg}}(\qc(X_1 \ti X_2)).
$$

Thus it remains to show that the above morphism of monads is an equivalence.
It is a straightforward diagram chase to check that the monad $T_{alg} = \tau_*\tau^*$
is nothing more than the composition $\pi_1^*\pi_{0*}$ of the geometric functors associated
to the initial cosimplicial maps
$$
\xymatrix{
 X_1\times X_2
 \ar@<+.5ex>[r]^-{\pi_0} \ar@<-.5ex>[r]_-{\pi_1} &
 X_1\times Y \times X_2 
 }$$
 Thus by base change, it is equivalent to the monad
 $T_{geom}=\pi_*\pi^*$.
This concludes the proof of the theorem.
\end{proof}

%


\begin{cor} \label{prop dualizable}
For $\pi:X\to Y$ any map of perfect stacks, $\qc(X)$ is
self-dual as a $\qc(Y)$-module.
\end{cor}

\begin{proof}
By Theorem \ref{Thomason tensor}, we have a canonical factorization
$\qc(X \times_Y X) \simeq \qc(X)\otimes_{\qc(Y)} \qc(X)$. Using this
identification, we can define the unit and trace by the
correspondences $u=\Delta_*\pi^*:\qc(Y) \to \qc(X \times_Y X)$ and
$\tau=\pi_*\Delta^*:\qc(X \times_Y X) \to \qc(Y)$, where $\Delta:X\to
X\times_Y X$ is the relative diagonal. We need to check that
the following composition is the identity:
$$
\xymatrix{ \qc(X)  \ar[r]^-{u \ot \on{id}} & \qc(X)\ot_{\qc(Y)} \qc(X)\ot_{\qc(Y)} \qc(X)
\ar[r]^-{\on{id}\ot\tau}  & \qc(X) }
$$
The argument is a chase in the following diagram (with Cartesian square):
$$
\xymatrix{
& 
\ar[d]_{\Delta} X \ar[r]^-{\Delta} 
& 
X\ti_Y X \ar[r]^-{\pi_1} \ar[d]^-{\Id_1\ti \Delta_{23}}
& 
X\\
X
& 
\ar[l]_-{\pi_2} X\ti_Y X \ar[r]^-{ \Delta_{12}\ti \Id_3}
& 
X \ti_{Y} X\ti_Y X
&
}
$$
Applying base change and identities for compositions, we have equivalences of functors
\begin{eqnarray*}
(\on{id}\ot\tau)\circ (u \ot \on{id}) & = & \pi_{1*} (\Id_1\ti \Delta_{23})^*(\Delta_{12}\ti \Id_3)_*\pi^*_{2}\\
& \simeq & \pi_{1*} \Delta_*\Delta^*\pi^*_{2}\\
& \simeq & \Id_{\qc(X)}
\end{eqnarray*}
\end{proof}

\begin{remark}
The above argument also shows that $\qc(X)^c$ is dualizable 
over perfect $k$-modules $\qc(\Spec k)^c$ when $X$ is smooth and proper: smoothness is required for the unit in $\qc(X\times X)$ to be
perfect and properness is required for the trace to land in perfect $k$-modules.
\end{remark}

For a derived stack $X$,
let $\cM, \cM'$ be stable presentable $\qc(X)$-modules.
To reduce notation, we write
$\Fun_X (\cM, \cM')$ for the stable presentable $\infty$-category of $\qc(X)$-linear colimit preserving
functors $\cM\to \cM'$.

\begin{cor} \label{transforms with air}
Let $X,X'$ and  $Y$ be perfect stacks with maps $X\rightarrow Y \leftarrow X'$. 
Then there is a natural equivalence of $\infty$-categories 
$$
\qc(X\times_Y X') \risom \Fun_Y (\qc(X) , \qc({X'})).
$$ 
In other words, the $\oo$-category
of integral kernels is equivalent to the
$\oo$-category of functors.
\end{cor}

\begin{proof}
The statement is an immediate consequence of Theorem \ref{Thomason tensor} and the fact that
$\qc(X)$ is self-dual (Corollary \ref{prop dualizable}). It implies that
internal hom of $\qc(Y)$-modules out of $\qc(X)$ is calculated by
tensoring with $\qc(X)^\vee\simeq\qc(X)$.
\end{proof}

\begin{remark}
The equivalence of Corollary \ref{transforms with air} is naturally
monoidal in the following sense. For perfect stacks $X,X',X''$
mapping to a perfect stack $Y$, there is a convolution map
$$\qc(X\times_Y X')\ot \qc(X'\times_Y X'') \to \qc(X\times_Y X'')$$
given by pulling back and pushing forward with respect to the triple
product $X\times_Y X' \times_Y X''$ (see Section \ref{convolution}).
On the other hand, we have a composition map $$\Fun_Y (\qc(X) ,
\qc(X'))\ot \Fun_Y(\qc(X'),\qc(X''))\to \Fun_Y(\qc(X),\qc(X'')),$$
and the equivalence of the theorem intertwines these composition maps.
\end{remark}

Finally, for a finite simplicial set  $\Sigma$, we would like to compare the formation of mapping stacks
$X^\Sigma =\Map(\Sigma, X)$ (this can be viewed
as the
cotensoring of stacks over simplicial sets) with the formation of
tensor products of $\infty$-categories $\cC\ot\Sigma$  (the tensoring of
symmetric monoidal $\infty$-categories over simplicial sets). By the notation
$\cC\ot\Sigma$, we mean the geometric realization of the simplicial
$\infty$-category given by the constant assignment of $\cC$ to each simplex
of the simpicial set $\Sigma$.

\begin{cor}\label{cor qc of finite mapping stacks}
Let $X$ be a perfect stack, and let $\Sigma$ be a finite simplicial
set.  Then there is canonical equivalence
$$\qc(X^\Sigma) \simeq\qc(X)\ot \Sigma.$$
In other words, the $\infty$-category of sheaves on the mapping stack is 
calculated as the tensor product of
 the $\infty$-categories of sheaves  on the simplices.
\end{cor}

\begin{proof}
We calculate $\qc(X^\Sigma)$ by induction on the simplices as an
iterated fiber product of copies of $X$ over perfect stacks (by
Proposition \ref{Thomason product}), and apply Theorem \ref{Thomason
tensor} at each stage.
\end{proof}



\subsection{General base stacks}\label{main proof}

In this section, we extend Corollary~\ref{transforms with air} 
on integral transforms 
to a relative setting where the base is allowed to be an arbitrary derived stack
with affine diagonal (not necessarily perfect).
In order to describe integral transforms relative to such a base, we
will utilize the simple behavior of $\infty$-categories of sheaves under affine base change.

\begin{prop}\label{lem more tensor products}
Let $Y$ be a derived stack with affine diagonal, and let $f: \Spec A \ra Y$ be
an affine over $Y$. Then $\Mod_A$ is a self-dual $\qc(Y)$-module.
In particular, for any $X\to Y$ there is a canonical
equivalence
$$\qc(X\times_Y \Spec A) \simeq \qc(X)\otimes_{\qc(Y)} \Mod_A.$$
\end{prop}

\begin{proof}
Since $Y$ has affine diagonal, the map $f:\Spec A\ra Y$ is a
relative affine, which implies that $f_*$ is colimit preserving and
conservative. Hence $f_*$ satisfies the monadic Barr-Beck criteria
\cite[Theorem 3.4.5]{dag2}, implying that the natural map $\Mod_A \ra \Mod_{f_*
f^*}(\qc(Y))$ is an equivalence. By the projection formula, the monad
$f_* f^*(-)$ is equivalent to the functor $f_*A \otimes (-)$, with
monad structure given by the algebra structure on $A$. As a
consequence, we see that $\Mod_A$ is equivalent to
$\Mod_{f_*A}(\qc(Y))$. Now Proposition~\ref{dualizable prop} gives
that $\Mod_{f_*A}(\qc(Y))$ is self-dual as a $\qc(Y)$-module.

Since the $\infty$-category $\Mod_A$ is a dualizable $\qc(Y)$-module it
follows that the functor 
$${(-) \otimes_{\qc(Y)} \Mod_A}$$ commutes with
limits of $\qc(Y)$-module categories. Thus we have equivalences
$$\qc(X)\ot_{\qc(Y)}\Mod_A \simeq
(\lim_{A'\in \Aff/X} \Mod_{A'}) \otimes_{\qc(Y)} \Mod_A
 \simeq \lim_{A'\in \Aff/X} (\Mod_{A'} \otimes_{\qc(Y)} \Mod_A).
$$
\noindent
Another application of Proposition~\ref{dualizable prop} implies the following equivalence
$$\lim_{A'\in \Aff/X} (\Mod_{A'} \otimes_{\qc(Y)} \Mod_A) \simeq \lim_{A'\in \Aff/X} (\Mod_{f'_*A' \ot f_*A}(\qc(Y))).$$Using that the map $f'\times f: \Spec A' \times_Y \Spec A \ra Y$ is affine and that $(f'\times f)_* (A'\boxtimes A) \simeq f'_*A'\otimes f_*A$, we obtain that the above is further equivalent to
$$\lim_{A'\in \Aff/X} (\qc(\Spec A' \times_Y \Spec A)) \simeq \qc( \colim_{A'\in \Aff/X} (\Spec A' \times_Y \Spec A))$$ 
$$\simeq \qc((\colim_{A' \in \Aff/X} \Spec A')\times_Y \Spec A)\simeq \qc(X\times_Y \Spec A)$$
\noindent
Here we have used that $\qc(-)$ sends all colimits to limits, and that $(-)\times_Y \Spec A$ commutes with colimits (since it is the left adjoint to the mapping stack over $Y$).
\end{proof}

We now show that in the
general setting where the base is an  arbitrary derived stack
with affine diagonal,
 functors continue to be given by integral kernels.

\begin{theorem} \label{main theorem}
Let $f:X \ra Y$ be a perfect map of derived stacks with affine diagonal, and let $g:X' \ra
Y$ be an arbitrary map of derived stacks. Then there is a natural
map $\qc({X\times_Y X'}) \ra \Fun_Y (\qc(X) , \qc({X'}))$ that is an
equivalence of $\infty$-categories.
\end{theorem}

\begin{proof} Consider the
Cartesian diagram
$$
\xymatrix{
X\times_Y X' \ar[r]^-{\tilde f} \ar[d]^-{\tilde g} & X' \ar[d]^-{g}\\
X \ar[r]^-f & Y\\}
$$

We define a functor $\qc({X\times_Y X'}) \ra \Fun (\qc(X) ,
\qc({X'}))$ by sending a quasi-coherent sheaf $M \in \qc(X \times_Y
X')$ to the functor $\tilde{f}_* (M\otimes \tilde{g}^* - )$. This is
a colimit preserving functor, since $\tilde f$ is perfect. Furthermore,
using 
the projection
formula for the map $\tilde f$, this functor naturally admits the
extra structure of $\qc(Y)$-linearity as follows
$$
\tilde{f}_* (M\otimes \tilde{g}^* -) \otimes g^* V \simeq
\tilde{f}_* (M\otimes \tilde{g}^* (-)  \otimes \tilde f^* g^* V) $$
$$\simeq \tilde{f}_* (M\otimes \tilde{g}^* (-)  \otimes \tilde g^* f^* V)
\simeq \tilde{f}_* (M\otimes \tilde{g}^* (- \otimes f^* V)),
$$
for any $V\in\qc(Y)$. Therefore, we in fact obtain a functor
$\qc(X\times_Y X') \ra \Fun_Y (\qc(X) , \qc(X'))$, and the rest of this proof
will be devoted to showing it is an equivalence.

Recall that the $\infty$-category of quasi-coherent sheaves on $X'$ is given by the limit
$$
\qc({X'}) \simeq \lim_{ \Aff/X'} \Mod_A.
$$
By working locally in the target,
this provides a description of the $\infty$-category of $\qc(Y)$-linear functors with values in $\qc({X'})$ as
the limit
$$
\Fun_Y (\qc(X) , \qc({X'})) \simeq \Fun_Y (\qc(X) , \lim_{\Aff/X'} \Mod_A) \simeq
\lim_{\Aff/X'} \Fun_Y (\qc(X) , \Mod_A)
$$

Likewise,
we have a description of the $\infty$-category of quasi-coherent sheaves on
the fiber product $X\times_Y X'$ as a limit
$$
\qc({X\times_Y X'}) \simeq \lim_{\Aff/X'} \qc({X\times_Y \Spec A}).
$$
This follows from the fact that the functor $\qc(-)$ takes colimits
to limits, and the fiber product functor $X\times_Y(-)$ commutes
with all colimits (because it has a right adjoint).

Now one can analyze the functor
$\qc({X\times_Y X'}) \ra \Fun_Y (\qc(X) , \qc(X'))$
by considering the terms in the above two limits. That is, to prove the theorem,
it suffices to prove it locally in the target $X'$: for any $\Spec A \ra X'$, we must show that the functor
$$
\qc({X\times_Y \Spec A}) \ra \Fun_Y (\qc(X) , \Mod_A)$$
is an equivalence.

We will prove this in two steps. First, we will deal with case that the base $Y$ is affine.
Afterward, we will use this case to deal with a general base $Y$.

So assume for the time being that $Y= \Spec B$. Then $X$ is a perfect stack
over $B$, and by Corollary~\ref{prop dualizable}, $\qc(X)$
is a self-dual $\Mod_B$-module. Thus we have equivalences
$$
\Fun_B (\qc(X) , \Mod_A) \simeq \Fun_B (\Mod_B , \qc(X)^\vee \otimes_B \Mod_A) \simeq
\qc(X) \otimes_B \Mod_A.
$$

By Proposition~\ref{lem more tensor products} we know that the
functor $\qc(-)$ takes affine base change to tensor product of
$\infty$-categories. Therefore we have an equivalence
$$\qc({X\times_{B} \Spec A}) \simeq \qc(X) \otimes_B \Mod_A.
$$
Putting together the above equivalences, we conclude that we have equivalences
$$
\qc({X\times_{B}  \Spec A}) \simeq
\qc(X) \otimes_B \Mod_A
\simeq
\Fun_B (\qc(X) , \Mod_A),
$$
This proves the theorem
when $Y$ is affine.

Working locally in the base $Y$, we will now use the above
discussion to prove the theorem in general. As above, since $\qc(-)$
and fiber products behave well with respect to colimits we can
calculate the $\infty$-category of quasi-coherent sheaves on the fiber product
$X\times_Y \Spec A$ as the limit
$$
\qc(X\times_Y \Spec A) \simeq \lim_{B\in \Aff/Y} \qc(X\times_Y \Spec
B \times_Y \Spec A).
$$

To calculate the $\infty$-category of functors, we use the following: by
Proposition~\ref{lem more tensor products}, the $\infty$-category $\Mod_A$ is
a dualizable $\qc(Y)$-module and hence the functor $ \Mod_A
\otimes_{\qc(Y)}(-) $ commutes with all limits. Therefore we have an
equivalence
$$\Mod_A \simeq \Mod_A \otimes_{\qc(Y)} (\lim_{B\in \Aff/Y} \Mod_B)
\simeq \lim_{B\in \Aff/Y} \Mod_A \otimes_{\qc(Y)} \Mod_B,
$$
and so in particular we obtain equivalences
$$
\Fun_Y (\qc(X) , \Mod_A) \simeq
\Fun_Y (\qc({X}) , \lim_{B\in \Aff/Y} \Mod_A \otimes_{\qc(Y)} \Mod_B)
$$
$$
\simeq
\lim_{B\in \Aff/Y}\Fun_Y (\qc({X}) ,  \Mod_A \otimes_{\qc(Y)} \Mod_B).
$$

By the adjunction between induction and restriction, and a repeated application
of Proposition~\ref{lem more tensor products}, we also have equivalences
$$
\Fun_Y (\qc({X}) , \Mod_A \otimes_{\qc(Y)} \Mod_B)
\simeq
\Fun_B (\qc(X) \otimes_{\qc(Y)} \Mod_B , \Mod_A \otimes_{\qc(Y)} \Mod_B).
$$
$$
\simeq
\Fun_B (\qc({X\times_Y \Spec B}) , \qc({\Spec A\times_Y \Spec B}))
$$

Finally, by the above discussion and the affine case of the theorem with base $\Spec B$,
we obtain the following chain of equivalences
$$
\qc({X\times_Y \Spec A})
\simeq
\lim_{B\in \Aff/Y} \qc({X\times_Y \Spec B\times_Y \Spec A})
$$
$$
\simeq
\lim_{B\in \Aff/Y} \Fun_B (\qc({X\times_Y B}) , \qc({A\times_Y \Spec B})) \simeq
\Fun_Y (\qc(X) , \Mod_A).
$$
Since we previously reduced the theorem to the case when $X'=\Spec A$,
this completes the proof.
\end{proof}

\section{Applications}\label{applications}
In this section, we apply the results of Section \ref{integral} to
calculate the centers and traces (and their higher
$\cE_n$-analogues) of symmetric monoidal $\infty$-categories of
quasi-coherent sheaves. We also
calculate  centers and traces of $\infty$-categories of linear endofunctors
(equivalently  by Corollary~\ref{transforms with air},
quasi-coherent sheaves on fiber products)
with monoidal structure given by composition (equivalently, convolution over the base).


\subsection{Centers and traces}\label{defining centers}
We begin with a discussion of the general notions of centers and traces
of associative algebra objects
in closed symmetric monoidal $\infty$-categories. This is a general version
of the approach to topological Hochschild homology developed in \cite{EKMM,shipley}.
We then calculate the center and trace of the symmetric monoidal $\oo$-category 
of sheaves $\qc(X)$ on a perfect stack $X$ (where we think of $\qc(X)$
as an associative algebra object in the closed symmetric monoidal $\oo$-category $\mathcal Pr^{\rm L}$
of presentable $\oo$-categories).

Let $\cS$ be a symmetric monoidal presentable $\infty$-category.
Recall that an associative
algebra structure on $A\in \cS$ is an object 
is equivalent to the structure of algebra over the $\cE_1$ operad.
We briefly recall the $\infty$-category versions of 
two familiar facts from classical algebra. 

First, there is the notion of the opposite associative algebra $A^{\rm op}\in \cS$.
For this, recall that there is a map of operads $\tau: \cE_1 \ra \cE_1$ determined by the action of the symmetric group $\Sigma_2$ on the two contractible subspaces of $\cE_1(2)$, and such that $\tau\circ \tau$ is equivalent to the identity. Given an $\cE_1$-algebra $A$, 
the pullback $\tau^* A$ is by definition the opposite $\cE_1$-algebra $A^{\rm op}$. 

Second, given two associative algebras $A, B\in \cS$,
their monoidal product $A\ot B\in \cS$ 
carries a natural
associative algebra structure. 
 For this, recall that given two algebras $A$, $B$ over any topological operad $\cO$, we have the structure of an $\cO\times \cO$-algebra on the monoidal product $A\ot B$. 
 Since the term-wise diagonal map gives a map of operads $\cO\ra \cO\times \cO$, 
 we obtain an $\cO$-algebra structure on $A\ot B$
 by restriction along the diagonal.
 
Furthermore, any associative algebra $A\in \cS$  is a left (as well as a right) module object over the associative algebra
$A\ot A^{\rm op}$ via left and right multiplication.

Now assume further that $\cS$ is a closed symmetric monoidal $\infty$-category.
Then we have internal hom objects
\cite[2.7]{dag2}, and given $A\ot A^{\rm op}$-modules $M,N$, we can define the $A\ot A^{\rm op}$-linear
morphism object $\intHom_{A\ot A^{\rm op}}(M, N)\in \cS$.
 Likewise, given left and right $A\ot A^{\rm op}$ modules
$M,N\in\cS$ we have a pairing $M\ot_{A\ot A^{\rm op}} N\in \cS$ defined by the two-sided
bar construction  \cite[4.5]{dag2} over $A\ot A^{\rm op}$.

\begin{definition} \label{def center}\
Let $A$ be an associative algebra object in a closed symmetric monoidal
$\oo$-category~$\cS$.

\begin{enumerate}

\item The derived center or Hochschild cohomology $\cZ(A)=\hh^*(A)\in \cS$
is the endomorphism object $\intEnd_{A\ot A^{\rm op}}(A)$ of $A$ as an $A$-bimodule.

\item The derived trace or Hochschild homology $\Tr(A)=\hh_*(A)\in \cS$ is 
the pairing object ${A\ot_{A\ot A^{\rm op}} A}$
 of $A$ with itself as an $A$-bimodule.

\end{enumerate} 
\end{definition}

In general, the center $\cZ(A)$ is again
an associative algebra object in $\cS$, and the trace $\Tr(A)$ is an $A$-module object in $\cS$.
Furthermore, 
$\cZ(A)$ comes with a canonical central morphism 
$$
\xymatrix{
\fz:\cZ(A)
\ar[r] &
A 
&&
F\ar@{|->}[r]
& 
F(1_A)
}
$$ 
while $\Tr(A)$ comes with a canonical trace morphism
$$
\xymatrix{
\ftr:A 
\ar[r] &
\Tr(A)
&&
A\ar@{|->}[r]
& 
A\ot_{A\otimes A} 1_A
}
$$ 
coequalizing left and right multiplication. 
When $A$ is in fact symmetric, $\cZ(A)$ and $\Tr(A)$ are again naturally symmetric
algebra objects in $\cS$, though the symmetric algebra structure on $\cZ(A)$
is different from its general associative algebra structure.

\subsubsection{Cyclic bar construction}\label{hochschild complexes}
It is very useful to have versions of the Hochschild chain and cochain
complexes which calculate the trace and center.
In the setting of an associative algebra object $A$ in a monoidal $\infty$-category
$\cS$, they will take the form of a simplicial object $\mathbf{N}^{cyc}_*(A)$ 
and cosimplicial object $\mathbf{N}_{cyc}^*(A)$ such that the geometric realization
 $\colim \mathbf{N}^{cyc}_*(A)$ is the trace $\Tr(A)$ and the
totalization  $\lim \mathbf{N}_{cyc}^*(A)$ is the center $\cZ(A)$. 
(Note that simplicial and cosimplicial 
objects here are taken in the $\infty$-categorical sense, and so the diagram identities
hold up to coherent homotopies as in, e.g., \cite{angeltveit}, where the cyclic bar
construction for $A_\infty$ H-spaces is developed.)

We construct the simplicial object $\mathbf{N}^{cyc}_*(A)$
and cosimplicial object $\mathbf{N}_{cyc}^*(A)$
 as follows. 
 Consider the adjunction
$$
\xymatrix{\cS\ar@/^1pc/[r]^-{F}&\m_{A}(\cS) \ar[l]^-{G}
}
$$
where $F(-) =A\ot -$ is the induction, and $G$ is the forgetful functor.
It determines a comonad $S \simeq FG$ acting on $\m_A(\cS)$.

As in classical algebra,
for any $A$-module $M$, the comonad $S$ provides a canonical augmented
simplicial object $C_*(M)$ with terms 
$$
C_{n-1}(M) \simeq (FG)^{n} (M) \simeq A^{\ot n} \ot M
$$ 
such that the geometric realization of $C_*(M)$ is naturally equivalent to $M$.
We will say that $C_*(M)$ is a simplicial resolution of $M$.

We can apply this technique to produce the familiar cyclic bar construction. 
Consider the special case of the above adjunction
 $$
\xymatrix{ \m_{A^{\rm op}}(\cS)\ar@/^1pc/[r]^-{F}&\m_{A\ot A^{\rm op}}(\cS) \ar[l]^-G
}
$$
where $F(-) =A\ot -$ again is the induction, and $G$ is the forgetful functor from $A$-bimodules to right $A$-modules. 
Using the comonad $S\simeq FG$, 
we obtain a simplicial resolution $C_*(A)$ of the $A$-bimodule $A$
whose terms 
$C_{n-1}(A)\simeq A^{\ot n+1}$ are $A$-bimodules which are free as left $A$-modules.

Now recall that the trace $\Tr(A)$ is defined by the self-pairing
$A\ot_{A\ot A^{\rm op}}A$. Since the tensor product commutes with colimits, in particular geometric realizations, we calculate
$$
\Tr(A) = A\ot_{A\ot A^{\rm op}}A \simeq A\ot_{A\ot A^{\rm op}}|C_*(A)|
\simeq |A\ot_{A\ot A^{\rm op}} C_*(A)|.
$$
Thus the geometric realization of $A\ot_{A\ot A^{\rm op}} C_*(A)$
calculates the trace $\Tr(A)$.

We write $\mathbf{N}^{cyc}_*(A)$
for the simplicial object $A\ot_{A\ot A^{\rm op}} C_*(A)$ and refer to it as the Hochschild
chain complex.
Since $A$ is free as a right $A$-module, the terms of the simplicial object $C_*(A)$ 
are free as $A$-bimodules. Thus we can evaluate the terms of the Hochschild
chain complex
$$
\mathbf{N}^{cyc}_n(A) = A\ot_{A\ot A^{\rm op}} C_n(A)
\simeq
A\ot_{A\ot A^{\rm op}} A^{\ot n+2} \simeq A^{\ot n+1}
$$ 
In particular, there are equivalences $\mathbf{N}^{cyc}_0(A) \simeq A$ and $\mathbf{N}^{cyc}_1(A) \simeq A\ot A$, and the two simplicial maps $A\ot A \ra A$ are 
the multiplication and the opposite multiplication of $A$. 

Similarly, recall 
that the center $\cZ(A)$ is defined by the endomorphisms
$\intEnd_{A\ot A^{\rm op}}(A)$. Since morphisms take colimits in the domain to limits, 
in particular geometric realizations to totalizations, we calculate
$$
\cZ(A) = \intEnd_{A\ot A^{\rm op}}(A)
\simeq  \intHom_{A\ot A^{ op}}(|C_*(A)|, A)
\simeq |\intHom_{A\ot A^{ op}}(C_*(A), A)|
$$
Thus the totalization of $\intHom_{A\ot A^{ op}}(C_*(A), A)$
calculates the center $\cZ(A)$.

We write $\mathbf{N}_{cyc}^*(A)$
for the cosimplicial object $\intHom_{A\ot A^{ op}}(C_*(A), A)$
 and refer to it as the Hochschild cochain complex.
As before, we can evaluate the terms of the
Hochschild
cochain complex
$$
\mathbf{N}_{cyc}^n(A) = \intHom_{A\ot A^{ op}}(C_n(A), A)
\simeq
\intHom_{A\ot A^{ op}}(A^{\ot n+2}, A)
\simeq
\intHom(A^{\ot n}, A). 
$$ 
In particular, there are equivalences $\mathbf{N}_{cyc}^0(A) \simeq A$ and $\mathbf{N}_{cyc}^1(A) \simeq \intHom(A, A)$, and the two cosimplicial maps $A\to \intHom(A, A)$ are induced by
the left and right multiplication of $A$.

\subsubsection{Centers and traces for monoidal $\oo$-categories}
We will apply the above constructions in the following setting.
We will always take $\cS$ to be the $\oo$-category
${\mc Pr}^{\rm L}$ of presentable $\oo$-categories (with morphisms left adjoints).
Then a monoidal presentable $\infty$-category $\cC$
is an associative algebra
object in $\mc Pr^{\rm L}$.
Thus we have the notion of its 
center (or Hochschild cohomology category) $\cZ(\cC)=\Fun_{\cC\otimes
\cC^{\rm op}} (\cC , \cC)\in \mc Pr^{\rm L}$, and its trace (or
Hochschild homology category) $\Tr(\cC) =\cC\ot_{\cC\ot \cC^{\rm op}} \cC\in \mc Pr^{\rm L}$.

%



Now we will specialize further to a geometric setting.
Let $X$ be a perfect stack, and take $\cC$
to be the presentable stable $\oo$-category $\qc(X)$ 
equipped 
with its (symmetric) monoidal tensor product.
To calculate the center and trace of $\qc(X)$,
we introduce the loop space
$$
\cL X=\Map(S^1, X) \simeq X \times_{X\times X} X
$$
where the fiber product is along two copies of the diagonal
map. 

\begin{cor} \label{cor center} Let $X$ be a perfect stack,
and equip $\qc(X)$ with its (symmetric) monoidal algebra structure given by tensor product.
Then there are canonical equivalences (of symmetric monoidal) $\infty$-categories
$$
\qc(\cL X)\simeq \cZ(\qc(X)) \simeq \Tr(\qc(X))
$$
\end{cor}

\begin{proof}

By Theorem \ref{Thomason tensor}, we know that $\Tr(\qc(X))$,
which is a tensor product, is also calculated by a fiber
product 
$$\Tr(\qc(X))\simeq
\qc(X\times_{X\times X} X).
$$ 

On the other hand, by
Corollary \ref{transforms with air}, we know that $\cZ(\qc(X))$, 
which consists of functors, is also calculated by a tensor
product
$$\cZ(\qc(X))\simeq \Tr(\qc(X)).
$$
\end{proof}






\subsection{Centers of convolution categories}\label{convolution}
Let $p:X\to Y$ be a map of perfect stacks satisfying descent. 
Consider the
convolution diagram
$$
\xymatrix{
& \ar[dl]_-{p_{12}}\ar[dr]^-{p_{23}} \ar[d]_{p_{13}}X \times_Y X \times_Y X & \\
\XYX & \XYX & \XYX }
$$
Equip $\qc(X\times_Y X)$ with the monoidal product defined by
convolution
$$
M\star N = p_{13*}(p_{12}^*(M) \ot p_{23}^*(N)).
$$
By Theorem~\ref{main theorem}, we have a monoidal equivalence
$$
\qc(\XYX)\simeq \Fun_Y(\qc(X), \qc(X)).
$$

Consider the fundamental correspondence
$$
\xymatrix{
 \cL Y  = Y\ti_{Y\ti Y} Y & \cL Y \ti_Y X =X \ti_{X \ti Y} X \ar[l]_-{\pi} \ar[r]^-{\delta} & X\ti_Y X.\\
}
$$
where the maps are defined by the formulas
$$
\pi = \Id_{\cL Y} \ti_{\Id_Y} p =  p \times_{p \times \Id_Y} p \qquad \delta = \Id_X
\times_{\pi_Y} \Id_X.
$$
where $\pi_Y:X\ti Y\to Y$ is the obvious projection.

Passing to sheaves, we obtain a diagram
$$
\xymatrix{
\qc(\cL Y)   \ar[r]^-{\pi^*}&\qc(\cL Y \ti_Y X)  \ar[r]^-{\delta_*} & \qc(X\ti_Y X)\\
}
$$
which by Theorem~\ref{main theorem}
admits the interpretation
$$
\xymatrix{ \Fun_{Y\ti Y}(\qc(Y), \qc(Y))  \ar[r]^-{\pi^*} &
\Fun_{X\ti Y}(\qc(X), \qc(X)) \ar[r]^-{\delta_*} &
\Fun_{Y}(\qc(X), \qc(X)) .\\
}
$$
where $\pi^*$ is the $\qc(X)$-linear induction, and $\delta_*$
forgets the $\qc(X)$-linear structure.

\medskip

The aim of this section is to prove the following.

\begin{theorem}\label{thm convolution} Suppose $p: X\to Y$ is
a map of perfect stacks satisfying descent.
Then there is a canonical equivalence
$$
\cZ(\qc(\XYX)) \simeq \qc(\cL Y)
$$
such that the forgetful functor $\fz:\cZ(\qc(\XYX)) \to \qc(\XYX)$ is given by the correspondence
$\delta_*\pi^*:\qc(\cL Y)\to \qc(\XYX)$.
\end{theorem}

The proof occupies the remainder of this section. We break up the argument
into a general discussion and then its specific application.

At the end of the section, we also explain an analogous description of traces,
conditional on a still to be developed version of Grothendieck duality
in the derived setting. Namely, assuming further that
$p:X\to Y$ is proper with invertible dualizing sheaf and
 Grothendieck duality holds,
we explain how to deduce an expected canonical equivalence
$$ \Tr(\qc(\XYX))\simeq \qc(\cL Y)$$ 
such that the trace
$\ftr:\qc(\XYX)\to \Tr(\qc(\XYX))$ is given by the
correspondence $\pi_*\delta^*:\qc(\XYX)\to \qc(\cL Y)$.



\subsubsection{Relative cyclic bar construction} 
For the proof of Theorem~\ref{thm convolution},
it will be useful to introduce relative versions of the Hochschild chain and cochain
complexes.
In general, the setup for the construction will be a map of 
associative algebra objects $B\to A$ in a monoidal $\infty$-category
$\cS$. The usual Hochschild chain and cochain
complexes introduced in the previous section will correspond
to the case when $B$ is the unit $1_\cS$. 

Consider the adjunction
 $$
\xymatrix{ \m_{B\ot A^{\rm op}}(\cS)\ar@/^1pc/[r]^-{F}&\m_{A\ot A^{\rm op}}(\cS) \ar[l]^-G
}
$$
where $F(-) =A\ot_B -$ is the induction, and $G$ is the forgetful functor from $A$-bimodules to 
$B\ot A^{\rm op}$-modules. 
Using the comonad $S\simeq FG$, 
we obtain a simplicial resolution $C^B_*(A)$ of the $A$-bimodule $A$
with terms 
$$
C^B_{n-1}(A)\simeq A \ot_B \ot \cdots \ot_B A,
\quad
\mbox{ with $n+1$ terms.}
$$
To reduce the notation, we will denote the above expression by $A_B^{\ot n+1}$.

As before, 
the geometric realization of $A\ot_{A\ot A^{\rm op}} C^B_*(A)$
calculates the trace $\Tr(A)$,
and
the totalization of $\intHom_{A\ot A^{ op}}(C^B_*(A), A)$
calculates the center $\cZ(A)$.
We write $\mathbf{N}^{B}_*(A)$
for the simplicial object $A\ot_{A\ot A^{\rm op}} C^B_*(A)$ and refer to it as the relative Hochschild
chain complex.
Similarly,
we write $\mathbf{N}_{B}^*(A)$
for the cosimplicial object $\intHom_{A\ot A^{ op}}(C^B_*(A), A)$
 and refer to it as the relative Hochschild cochain complex.

Continuing as before, we can evaluate the terms of the relative Hochschild
chain complex
$$
\mathbf{N}^{B}_n(A) = A\ot_{A\ot A^{\rm op}} C^B_n(A)
\simeq
A\ot_{A\ot A^{\rm op}} A_B^{\ot n+2} 
\simeq
B\ot_{B\ot B^{\rm op}} A_B^{\ot n+1} 
$$ 
Similarly,
we can evaluate the terms of the relative
Hochschild
cochain complex
$$
\mathbf{N}_{B}^n(A) = \intHom_{A\ot A^{ op}}(C^B_n(A), A)
\simeq
\intHom_{A\ot A^{ op}}(A_B^{\ot n+2}, A)
\simeq
\intHom_{B\ot B^{\rm op}}(A_B^{\ot n+1}, B). 
$$


\subsubsection{Proof of Theorem \ref{thm convolution}}
Now we will apply the preceding formalism to calculate the center
of the monoidal $\infty$-category $\qc(\XYX)$. Recall that our aim is
to show that there is a canonical equivalence 
$$
\cZ(\qc(\XYX)) \simeq
\qc(\cL Y).
$$
where $\cL Y$ is the loop space of $Y$.

Consider the 
symmetric monoidal $\infty$-category $\qc(X)$, together with the monoidal
functor 
$$\Delta_*:\qc(X)\to \qc(\XYX)
$$ obtained via pushforward
along the relative diagonal $\Delta:X\to \XYX$.

Set $A= \qc(\XYX)$, and $B = \qc(X)$.
By the preceding discussion, the center $\cZ(A)$
is the limit of the relative Hochschild cochain complex $\mathbf{N}^*_B(A)$.
Furthermore, its terms can be calculated
$$
\mathbf{N}_{B}^n(A) 
\simeq
\intHom_{B\ot B^{\rm op}}(A_B^{\ot n+1}, B). 
$$ 
Applying the results of Section~\ref{integral}, we can
rewrite each term in the form
$$
\mathbf{N}_{B}^n(A) \simeq \qc(\cL Y \ti_Y X \ti_Y\cdots\ti_Y X),
\quad\mbox{ with $n+1$ copies of $X$.}
$$
Here we have used the elementary identification
$$
\cL Y \ti_Y X \ti_Y\cdots\ti_Y X
 \simeq
X \times_{X \ti X}((\XYX) \ti_X \cdots \ti_X (\XYX))
$$
where the left hand side has $n+1$ copies of $X$, and the right
hand side has $n+1$ copies of $\XYX$.

We conclude that the terms of $\mathbf{N}_{B}^*(A)$ are nothing more than the
terms of the cosimplicial $\infty$-category obtained by applying 
$\qc(-)$
to the \u Cech simplicial stack induced by the map 
$$
\tilde p:\cL Y \ti_Y X\to \cL Y
$$
obtained by base change from the original map $p:X\to Y$. It
is straightforward to check that under this identification the coboundary maps
are given by the usual \u Cech pullbacks. By assumption, $p$ satisfies descent,
hence $\tilde p$ satisfies descent, and thus the totalization
$\lim \mathbf{N}_{B}^*(A)$  also calculates the $\infty$-category $\qc(\cL Y)$.

Finally, note that under this identification,
the composition $\delta_*\pi^*:\qc(\cL Y)\to \qc(\XYX)$ corresponds
to the composition 
$$
\qc(\cL Y)\simeq \lim \mathbf{N}_{B}^*(A) \to \mathbf{N}_{B}^0(A) \to \mathbf{N}_{cyc}^0(A)  
\simeq \qc(\XYX)
$$
which is precisely the central functor $\fz$.
This concludes the proof of Theorem~\ref{thm convolution}.


\subsubsection{Traces and Grothendieck duality}
Finally, we explain here an analogous description of traces,
conditional on a still to be developed version of Grothendieck duality
in the derived setting. Namely, assuming further that
$p:X\to Y$ is proper with invertible dualizing sheaf and
 Grothendieck duality holds,
we explain how to deduce an expected canonical equivalence
$$ \Tr(\qc(\XYX))\simeq \qc(\cL Y)$$ 
such that the trace
$\ftr:\qc(\XYX)\to \Tr(\qc(\XYX))$ is given by the
correspondence $\pi_*\delta^*:\qc(\XYX)\to \qc(\cL Y)$.

We continue with the notation from the proof of Theorem~\ref{thm convolution}
and the preceding sections.
We will show that the geometric realization 
$\colim \mathbf{N}^B_*(A)$ of the relative Hochschild chain complex also
calculates the $\infty$-category $\qc(\cL Y)$.
As before, applying the results of Section~\ref{integral}, we can
rewrite the terms of $ \mathbf{N}^B_*(A)$ in the form
$$
\mathbf{N}^B_n(A) \simeq \qc(\cL Y \ti_Y X \ti_Y\cdots\ti_Y X)
\quad
\mbox{ with $n+1$ copies of $X$.}
$$
Furthermore, 
it is straightforward to check that under this identification the boundary maps
of $\mathbf{N}^B_*(A)$
are given by the pushforwards $\mathfrak p_{n*}$ which are right adjoints to the usual \u Cech pullbacks
$\mathfrak p_n^*$.

To reduce notation, set $X_n =  X \ti_Y\cdots\ti_Y X$ with $n+1$ copies of $X$.

Now suppose that $p:X\to Y$ is proper and
has an invertible dualizing complex (Gorenstein). 
Then we expect Grothendieck duality to hold in the following form:
the pushforwards $\mathfrak p_{n*}$ are also
{left} adjoints to the pullbacks 
$$
\mathfrak p_n^!(-) \simeq \mathfrak p_n^*(- \otimes \omega^{-1}_{\cL Y \ti_Y X_{n-1}/Y}) \otimes \omega_{\cL Y \ti_Y X_n/Y},
$$ 
where $\omega_{\cL Y \ti_Y X_n/Y}$ denotes the relative dualizing sheaf of 
$\cL Y \ti_Y X_n \to Y$.
Under this assumption, we find that the limit
$\qc(\cL Y)$ of the 
cosimplicial $\infty$-category $(\qc(\cL Y \ti_Y X_n), \mathfrak p_n^*)$
admits the following alternative description. 

First, we can identify the cosimplicial $\infty$-category $(\qc(\cL Y \ti_Y X_n), \mathfrak p_n^*)$ 
with the 
cosimplicial $\infty$-category $(\qc(\cL Y \ti_Y X_n, \mathfrak p_n^!)$
via tensoring by the inverse of the relative dualizing sheaf $\omega_{\cL Y \ti_Y X_n/Y}$ on each simplex.
In particular, we obtain an identification of their limits.

Second, we can consider the cosimplicial $\infty$-category $(\qc(\cL Y \ti_Y X_n), \mathfrak p_n^!)$
as a diagram
in the $\infty$-category $\mc Pr^{\rm R}$ of presentable $\infty$-categories (with morphisms right adjoints).
By \cite[Theorem 5.5.3.18]{topos},
the calculation of the limit of the diagram does not depend on this choice
of context.
Then we can pass to the opposite $\infty$-category $\mc Pr^{\rm L}$  of presentable $\infty$-categories (with morphisms left adjoints). To calculate a limit in $\mc Pr^{\rm R}$
is the same as to calculate a colimit in $\mc Pr^{\rm L}$.
But we have seen that the trace $\Tr(\cH)$
is precisely the colimit of the dual simplicial diagram $(\qc(\cL Y \ti_Y X_n), \mathfrak p_{n*})$.
Thus we conclude that the geometric realization $\colim  \mathbf{N}_*^{B}(A)$
also calculates the $\infty$-category $\qc(\cL Y)$.

Finally, note that under this identification,
the composition $\pi_*\delta^*: \qc(\XYX) \to \qc(\cL Y)$ corresponds
to the composition 
$$
\qc(\XYX) \simeq  \mathbf{N}_0^{cyc}(A) \to   \mathbf{N}_0^{B}(A) \to  \colim \mathbf{N}_*^{cyc}(A)
\simeq \qc(\cL Y)
$$
which is precisely the trace~$\ftr$.



\subsection{Higher centers}\label{sect higher centers}
Unlike in classical algebra, for an
algebra object in an $\oo$-category, commutativity is an additional structure rather than a property. 
More precisely, the forgetful functor from $\cE_\infty$-algebras to $\cE_1$-algebras is 
conservative but not fully faithful. 
Given an $\cE_\infty$-algebra $A$ in a symmetric monoidal $\oo$-category $\cS$,
the center $\cZ(A)$ that we have studied to this point does not involve the commutativity of $A$.
More precisely, it only depends upon the $\cE_1$-algebra underlying $A$.

In this section, we discuss higher versions of the center where we
forget less commutativity.  For a perfect stack $X$, we calculate the
$\cE_n$-center of the $\cE_n$-category underlying the
$\cE_\oo$-category $\qc(X)$.  A detailed treatment of the foundations
of the subject may be found in Chapter 2 of \cite{thesis}. Here we
summarize only what is required for our consideration of the
$\cE_n$-center of $\qc(X)$.

\begin{remark} As noted in the introduction, after the completion of this manuscript, 
the paper \cite{dag3} was revised to include a thorough treatment of
$\oo$-categorical operads and their algebras, and the paper
\cite{dag6} treats in great detail the specific case of the
$\cE_n$-operads. We refer the reader to these preprints for further
details.
\end{remark}

Let $F$ be a topological operad. We define a topological category
$\cF$ with objects finite pointed sets and morphism spaces given by
$$
\Map_\cF(J_*, I_*) = \coprod_{f:J_* \ra I_*} \prod_I F(f^{-1}\{i\}).
$$
We then obtain an $\oo$-category, also denoted by $\cF$,
 by applying the singular functor to the mapping spaces to get a simplicial category and then taking the simplicial nerve.
 
\begin{definition}[\cite{thesis}] An $\cF$-monoidal structure on an $\oo$-category $\cC$ consists of a functor $p: \cF \ra {\rm Cat}_{\oo}$ with an identification $\cC \simeq p(1_*)$ such that the natural maps 
$p(J_*) \ra \prod_J p(1_*)$ resulting from a choice of a point $x\in F(J)$ is an equivalence.
\end{definition}

We will refer to the $\infty$-category $\cC$ equipped with an
$\cF$-monoidal structure as an $\cF$-category. This construction will
be of particular interest to us when $\cF$ is the little $n$-disks
operad $\cE_n$.

To discuss Hochschild cohomology, we must next introduce the notion of operadic modules. 

To this end, let ${{\rm Fin}_*}$ be the category of finite based sets. 
Let ${\rm Fin}_+$ be the category of ``doubly-based sets" 
with objects finite based sets $I_*$ and $(I\amalg +)_*$, 
and morphisms $\alpha$ of the underlying based sets such that
$+\in \alpha^{-1}(+)$, and either $\alpha(+) = +$ or $\alpha(+) = *$. 

Given a topological operad $F$, recall the $\oo$-category $\cF$ introduced above.
We also define the  $\oo$-category $\cF_+$ to be the fiber product of $\oo$-categories 
$$
\cF_+ = \cF\times_{{\rm Fin}_*} {\rm Fin}_+.
$$

\begin{definition} Let $\cC$ be an $\cF$-category. An $\cF$-$\cC$-module structure on an $\oo$-category $\cM$ is a functor $q: \cF_+ \ra {\rm Cat}_{\oo}$ extending the $\cF$-monoidal structure on 
$\cC$, together with an identification $q(+) \simeq \cM$ such that the natural map $q((I\amalg +)_*) \ra \cC^I \times \cM$ is an equivalence for any $I$.
\end{definition}

\begin{example} When $\cF$ is the $\cE_1$ operad, the notion of an $\cF$-$\cC$-module coincides
with that of a $\cC$-bimodule: it is an $\oo$-category left and right tensored over $\cC$. When $\cF$ is the $\cE_\infty$ operad, the notion coincides with that of a left (or equivalently right) module.
\end{example}

There is a similar definition of an $\cF$-algebra and an $\cF$-$A$-module in a symmetric monoidal or $\cF$-monoidal $\oo$-category $\cM$ generalizing the definition given above. This requires the notion of an $\cF$-lax monoidal functor (see \cite[Chapter 2, Definition 3.10]{thesis}). The simplification of the previous definition is available because ${\rm Cat}_\infty$ is equipped with the Cartesian monoidal structure. We will avail ourselves of this greater generality in the following definition, although the only case that will concern us in the following is when $\cM$ is 
either ${\rm Cat}_\infty$ or $\mc Pr^{\rm L}$.

Now let $\cM$ be a presentable symmetric monoidal
$\oo$-category whose monoidal structure distributes over colimits. 
Let $A$ be an $\cF$-algebra in $\cM$, and let $\m_A^\cF(\cM)$ be the $\oo$-category of $\cF$-$A$-modules in $\cM$ (see  \cite[Chapter 2, Definition 4.3]{thesis}). Under the above assumptions, 
the $\infty$-category $\m_A^\cF(\cM)$ is naturally tensored over $\cM$.

\begin{definition} 
For an $\cF$-algebra $A$ in $\cM$, we define the $\cF$-Hochschild cohomology 
$$
\hh^*_\cF(A) =\intHom_{\m_A^\cF}(A,A)
$$ 
to be the object of $\cM$ representing the endomorphisms of $A$ as an object of $\m_A^\cF(\cM)$.
\end{definition}

In particular, when $\cM = \mc Pr^{\rm L}$, 
the $\cF$-Hochschild cohomology of an $\cF$-category $\cC$ is the $\oo$-category of $\cF$-$\cC$-module functors 
$$
\hh^*_\cF(\cC) = \Fun_{\m_\cC^\cF}(\cC, \cC).
$$
%

We now specialize to the case where $\cF$ is the $\cE_n$-operad. Intuitively, an $\cE_n$-algebra structure on an object $A$ is equivalent to a family of associative algebra structures on $A$
parameterized by $S^{n-1}$  such that antipodal points are associated to opposite multiplications on 
$A$. An $\cE_n$-$A$-module structure on $M$ admits a similar intuitive interpretation
as a family of compatible left $A$-module structures on $M$
parameterized by $S^{n-1}$. This intuition leads to the following (to appear in \cite{inprep}, see also \cite{dag6}). 

\begin{prop}[\cite{inprep}]
Let $\cC$ be a presentable symmetric monoidal $\oo$-category 
whose monoidal structure distributes over colimits.

To an $\cF$-algebra $A$ in $\cC$ there is functorially assigned associative algebra $U_A$ such that there is a canonical equivalence $\m_A^\cF(\cC) \simeq \m_{U_A}(\cC)$ between $\cF$-$A$-modules and left $U_A$-modules. 

If $\cF$ is the $\cE_n$ operad, and the $\cE_n$-algebra structure on $A$ is obtained by restriction from an $\cE_\infty$-algebra structure, then there is a canonical equivalence of 
associative algebras $U_A \simeq S^{n-1}\ot A$.
\end{prop}
%

For $n=1$, the proposition 
reduces to the familiar statement that for an $\cE_\infty$-algebra $A$,
an $A\ot A^{\rm op}$-module structure (in the form of an $\cE_1$-$A$-module structure)
is equivalent to an $A\ot A$-module structure (in the form of a left $U_A$-module structure). 

For our current purposes, one can interpret the proposition as furnishing the definition of an 
$\cE_n$-$A$-module.
Namely, the reader uncomfortable with the abstractions can take
left $S^{n-1}\ot A$-modules as the definition of $\cE_n$-$A$-modules

\begin{example} Let $\cC$ be an $\cE_1$-algebra in $\mc Pr^{\rm L}$
(so $\cC$ is a presentable monoidal $\oo$-category whose monoidal structure distributes over colimits). Then left modules for the monoidal $\oo$-category $U_\cC$ are equivalent to $\cC$-bimodules. 
In particular, we have an equivalence $U_\cC \simeq \cC \ot \cC^{\rm op}$, where here $\cC^{\rm op}$ denotes the $\oo$-category $\cC$ equipped with the opposite monoidal structure. As a consequence, we see that in this case,
the preceding general definition of $\cF$-Hochschild cohomology recaptures the notion of the Drinfeld center introduced earlier.

\end{example}

Let $A$ be an $\cE_\infty$-algebra and $U_A$ be as above. We have a companion definition of $\cE_n$-Hochschild homology.

\begin{definition} 
For an $\cE_\oo$-algebra $A$,
we define the $\cE_n$-Hochschild homology 
$$
\hh_*^{\cE_n}(A) = A\ot_{U_A} A
$$ to be the tensor product
of $A$ with itself over the algebra $U_A$.
\end{definition}

We now have the following contribution of this paper to the story.

\begin{cor}
 \label{cor higher center}
For a perfect stack $X$, consider the stable $\oo$-category 
$\qc(X)$ equipped with its $\cE_n$-tensor product. Then
with $X^{S^n}=\Map(S^n, X)$,
there are canonical equivalences
$$
\qc(X^{S^n})\simeq \hh^*_{\cE_n}(\qc(X)) \simeq
\hh_*^{\cE_n}(\qc(X))
$$
\end{cor}

\begin{proof}
The result follows from an inductive application of Theorem \ref{Thomason
tensor} and Corollary \ref{cor qc of finite mapping stacks} to the
Cartesian diagrams $$ \xymatrix{
X^{S^n} \ar[r]\ar[d]&X\ar[d]\\
X\ar[r]& X^{S^{n-1}}\\}
$$
where the two maps $X \to X^{S^{n-1}}$ assign to a point of
$X$ the corresponding constant map.
\end{proof}


\section{Epilogue: Topological Field Theory}\label{TFT}

For a perfect stack $X$, the $\infty$-category $\qc(\cL X)$ of sheaves on
its loop space carries a rich structure coming from topological
field theory, generalizing the braided tensor category structure on
the usual Drinfeld center and providing a categorified analogue of
the Deligne conjecture on Hochschild cochains. In this section, we
discuss these structures and their generalizations.

\subsection{Topological Field Theory from perfect stacks}
Let $2{\rm Cob}$ denote the $\infty$-category whose objects are
compact oriented $1$-manifolds, and whose morphisms $2{\rm
Cob}(C_1,C_2)$ consist of the classifying spaces of oriented
topological surfaces with fixed incoming and outgoing components
$C_1,C_2$. This $\infty$-category has a symmetric monoidal structure given by
disjoint union of 1-manifolds. We will abuse notation by denoting
the object consisting of the disjoint union of $m$ copies of $S^1$
by $m$.

\begin{definition} A 2-dimensional topological field theory
valued in a symmetric monoidal $\infty$-category $\cC$ is a
symmetric monoidal
functor $F:2{\rm Cob}\to\cC$.
\end{definition}

Fix a base commutative derived ring $k$.

\begin{definition} Let ${\rm Cat}_\infty ^{\rm ex}$ 
denote the $(\infty, 2)$-category of stable, presentable $k$-linear
$\infty$-categories, with 1-morphisms consisting of continuous exact
functors.
\end{definition}

Note that ${\rm Cat}_\infty ^{\rm ex}$ has a symmetric monoidal
structure, the tensor product of presentable stable categories as studied in
\cite[4.2]{dag2}, and the unit of
the monoidal structure is the $\infty$-category of $k$-modules. Since the
empty $1$-manifold is the unit of $2{\rm Cob}$, a topological field
theory $\cZ$ sends its endomorphisms, which are closed surfaces
$\Sigma$, to endomorphisms of the unit of the target $\infty$-category $\cC$.
When $\cC={\rm Cat}_\infty ^{\rm ex}$, the unit is $\Mod_k$ and
$\cZ(\Sigma)$ is a $k$-module.

\begin{prop} For a perfect stack $X$, there is a
2d TFT $\cZ_X:2{\rm Cob}\to{\rm Cat}_\infty ^{\rm ex}$ with
$\cZ_X(S^1)=\qc(\cL X)$ and $\cZ_X(\Sigma)=\Gamma(X^\Sigma, \cO_{X^\Sigma})$,
for closed surfaces $\Sigma$.
\end{prop}

\begin{proof}
We define $\cZ_X$ on objects by assigning $\cZ_X(m)=\qc((\cL X)^{m})$.

To define $\cZ_X$ on morphisms, observe first that
 since $X$ is perfect, $(\cL X)^{m}$ is perfect,
and for $\Sigma\in 2{\rm Cob}^\circ(m,n)$,
the mapping stack $X^\Sigma=\Map(\Sigma, X)$ is also perfect
(special cases of Corollary~\ref{cor finite mapping stacks have air}).
Now for each $\Sigma\in 2{\rm Cob}^\circ(m,n)$, consider the correspondence
$$\xymatrix{
& \ar[dl] \ar[dr] X^\Sigma & \\
(\cL X)^{m}& & (\cL X)^{n}\\
}
$$
Since all of the stacks involved are perfect,
pullback and pushforward of quasi-coherent sheaves along this
correspondence defines a colimit-preserving functor
$$\cZ_X(\Sigma):\qc((\cL X)^{m})\to \qc((\cL X)^{n}).
$$
Thus applying this construction in families, we
obtain a map of spaces
$$\cZ_X(m,n):2{\rm Cob}^\circ(m,n)\to
\Fun(\qc((\cL X)^{ m}),\qc((\cL X)^{n})).$$

Suppose now that
$\Sigma=\Sigma_1 \coprod_{\coprod_k {S^1}} \Sigma_2 \in 2{\rm
Cob}(m,n)$ is obtained by sewing
two surfaces $\Sigma_1\in 2{\rm Cob}(m,k)$ and $\Sigma_2\in 2{\rm
Cob}(k,n)$. Then we have a diagram of correspondences
$$\xymatrix{
&&X^\Sigma \ar[dl]\ar[dr]&&\\
& \ar[dl] \ar[dr] X^{\Sigma_1}&&X^{\Sigma_2} \ar[dl] \ar[dr]\\
(\cL X)^{m}& & (\cL X)^{k} & &(\cL X)^{n}
}
$$
Since all of the stacks involved are perfect,
base change provides a canonical equivalence of functors
$$\cZ_X(\Sigma)\simeq \cZ_X(\Sigma_2)\circ
\cZ_X(\Sigma_1): \qc((\cL X)^{m})\to \qc((\cL X)^{n}).$$
Similar diagrams define the higher compositions.

To complete the construction,
note that $\cZ_X$ comes equipped with a canonical symmetric monoidal structure.
Namely, by Theorem~\ref{Thomason tensor}, there is a canonical equivalence
$$
\cZ_X(m) = \qc((\cL X)^{m})
\simeq  \qc( \cL X)^{\ot m} = \cZ_X(1)^{\ot m},
$$
and it clearly extends to a
symmetric monoidal structure.
\end{proof}

As an example, take $X=BG$
(in characteristic zero).
Then $\cZ_{BG}(S^1)$ is the $\infty$-category of sheaves on
the adjoint quotient $G/G={\rm Loc}_G(S^1)$, while
$\cZ_{BG}(\Sigma)$ is the cohomology of the structure sheaf of the
moduli stack $BG^{\Sigma}={\rm Loc}_G(\Sigma)$ of $G$-local systems
on $\Sigma$.

\medskip

More generally, the particular topological field theory operations
we are considering are not sensitive to the structure of manifolds,
and the same construction provides a field theory living over
``bordisms" of spaces. In what follows, by a space we will mean a space
homotopy equivalent to a finite simplicial set.

\begin{definition}
Let ${\rm Bord}_{\rm Spaces}$ denote the bordism $(\infty,1)$-category
of spaces, with 1-morphisms from $U$ to $V$ given by spaces $T_{UV}$
with maps $U\rightarrow T_{UV} \leftarrow V$, and composition of
$T:U\to V$ and $T':V\to W$ given by the homotopy pushout of spaces
$$T'\circ T= T \amalg_V T'.$$ We endow ${\rm Bord}_{\rm Spaces}$ with a symmetric monoidal
structure given by disjoint union.
\end{definition}

The proof of the previous proposition immediately extends to give the following:

\begin{prop} For any perfect stack $X$, there is a symmetric
monoidal functor $$S_X:{\rm Bord}_{\rm Spaces} \to {\rm Cat}_\infty
^{\rm ex}$$ with $S_X(U)=\qc(X^U)$ and $S_X(T:U\to V)$ the functor
$\qc(U)\to \qc(V)$ given by pullback and pushforward along the
correspondence $$X^U\leftarrow X^T \rightarrow
X^V.$$
\end{prop}

The 2d topological field theory $\cZ_X$ above is a categorified
analogue of the 2d TFTs defined by string topology on a compact
oriented manifold, or the topological B-model defined by a
Calabi-Yau variety \cite{costello}. Namely, we assign to the circle
the $\infty$-category of quasi-coherent sheaves, rather than the complex of
chains, on the loop space. This has the advantage that we may
construct maps for arbitrary correspondences without the assumption
that $X$ is a Calabi-Yau or satisfies Poincar\'e duality. On the
other hand, we have gone up a level of categoricity, pushing off the
problems of duality and orientation to the definition of operations between
$\Gamma(X^\Sigma, \cO_{X^\Sigma})$ for cobordisms between surfaces or invariants
for 3-manifolds. Recall \cite{BakalovKirillov} that an extended
3-dimensional topological field theory assigns a braided (in fact
ribbon) category to the circle. This category must however satisfy a
strong finiteness and nondegeneracy condition (modularity) coming
from its extension to three-manifolds. It would be interesting to
see what conditions need to be imposed on the sheaves we consider
and on the stack $X$ in order to extend $\cZ_X$ in such a fashion.
Note that Corollary \ref{prop dualizable} shows that $\qc(X)$ for $X$ perfect 
satisfies an analogue of the Calabi-Yau or Frobenius
property: namely it is self-dual as a $\Mod_k$-module.


\subsection{Deligne-Kontsevich conjectures for derived centers}
The notion of Drinfeld center for monoidal stable categories is a
categorical analogue of Hochschild cohomology of associative (or
$A_\infty$) algebras (or more precisely of Hochschild cochains, or
its spectral analogue, topological Hochschild cohomology). In the
case of algebras, Deligne's conjecture states that the Hochschild
cochain complex has the structure of an $\cE_2$-algebra (lifting the
Gerstenhaber algebra, or $H_*(\cE_2)$-algebra, structure on
Hochschild cohomology).  There is also a cyclic version of the
conjecture which states that the Hochschild cochains for a {\em
Frobenius} algebra possesses the further structure of a framed
$\cE_2$, or ribbon, algebra. See
\cite{T1,K,KS,McClureSmith,costello, KS Ainfty} for various proofs
of the Deligne conjecture and \cite{Kaufmann, TZ} for its cyclic
version. The Kontsevich conjecture (see \cite{T2, HKV}) generalizes
this picture to higher algebras, asserting that the Hochschild cochains
on an $\cE_n$-algebra have a natural $\cE_{n+1}$-structure.

The work of Costello \cite{costello} and Kontsevich-Soibelman
\cite{KS Ainfty} explains how a strong form of the Deligne
conjecture follows from the structure of topological field theory
-- specifically, by constructing a topological field theory $\cZ_A$ from
an algebra $A$ such that $\cZ_A(S^1)$ is the Hochschild cochain complex of
$A$.
Observe that the configuration space of $m$-tuples of circles
in the disk (the $m$-th space of the $\cE_2$-operad) maps to the
space of functors $\cZ(S^1)^{\ot m}\to \cZ(S^1)$ compatibly with

compositions.

\begin{prop}
In any 2d TFT $\cZ$, the object $\cZ(S^1)\in \cC$ is an algebra over
the framed $\cE_2$-operad (in particular, over the $\cE_2$-operad).
\end{prop}

A natural categorified analogue of the Deligne conjecture states
that the Drinfeld center of a monoidal $\infty$-category is an
$\cE_2$-category. For dualizable and self-dual
$k$-linear $\infty$-categories, one can ask if this structure comes from a
natural framed $\cE_2$, or ribbon, structure. We observe that our
identification of the Drinfeld center of $\qc(X)$ with sheaves on
the loop space, combined with the construction of the topological
field theory $\cZ_X$, automatically solves this conjecture in the
applicable cases:

\begin{cor}
For $X$ a perfect stack, the center
$\Tr(\qc(X))\simeq \qc(\cL X)$
has the structure of a stable framed $\cE_2$-category. 
For a map $X\to Y$ of perfect stacks, the same
holds for the center $\Tr(\qc(X\ti_Y X))\simeq \qc(\cL
X)$.
\end{cor}

In fact, the same arguments prove a case of the categorified form of
the Kontsevich conjecture. We consider $\qc(X^{S^n})$ as part of an
$n+1$-dimensional topological field theory, with operations given by
cobordisms of $n$-manifolds.

\begin{cor}
For $X$ a perfect stack, the $\cE_n$-Hochschild cohomology of the
stable $\cE_n$-category $\qc(X)$ has the structure of a stable
(framed) $\cE_{n+1}$-category.
\end{cor}



\end{document}